\documentclass[a4paper,11pt]{article}
\usepackage[T2A]{fontenc}
\usepackage[cp1251]{inputenc}
\usepackage[ukrainian, english]{babel}
\usepackage{amssymb,amsfonts,amsthm,amsmath,mathtext,cite,enumerate,float}
\usepackage[onehalfspacing]{setspace}
\usepackage{geometry}
\usepackage{graphicx}
\usepackage{multirow}
\usepackage{indentfirst}
\usepackage{color}

\newcommand{\aut}[1][n]{B_{#1}}
\newcommand{\syl}[1][n]{G_{#1}}

\newcommand{\m}{^{-1}}

\newcommand{\rest}[1][n]{_{(#1)}}

\newtheorem{theorem}{Theorem}
\newtheorem{proposition}[theorem]{Proposition}
\newtheorem{corollary}[theorem]{Corollary}
\newtheorem{lemma}[theorem]{Lemma}
\newtheorem{example}[theorem]{Example}
\newtheorem{remark}[theorem]{Remark}

\theoremstyle{definition}

\begin{document}
%\title{Commutator of Sylow 2-subgroups of alternating groups and symmetric group and their description}
%\title{Commutator and description of Sylow 2-subgroups of alternating groups and symmetric group}
\title{ Commutators of Sylow subgroups of alternating and symmetric groups, commutator width in the wreath product of groups}
%\title{The commutator width in the wreath product and commutator width of Sylow subgroups of alternating and symmetric groups}
\author{Ruslan Skuratovskii} %,
\date{\empty}
\maketitle

%\begin{center}
%\parbox{11.5cm}
%{\bf {Description and commutator
%} \/}

%\title[] {Description and commutator of Sylow 2-subgroups of alternating groups and symmetric group}
%\author[]{Ruslan Skuratovskii,  }

  \begin{abstract}
  %\note{Rewrite abstract.}
  This paper investigate bounds of the commutator width \cite {Mur} of a wreath product of two groups.
      The commutator width of direct limit of wreath product of cyclic groups are found.
    For given a permutational wreath product sequence of cyclic groups we investigate its commutator width and some properties of its commutator subgroup.
      It was proven that the commutator width of an arbitrary element of the wreath product of cyclic groups $C_{p_i}, \, p_i\in \mathbb{N} $ equals to 1. As a
corollary, it is shown that the commutator width of Sylows $p$-subgroups of symmetric and alternating groups $p \geq 2$ are also equal to 1.
The structure of commutator and second commutator of Sylows $2$-subgroups of symmetric and alternating groups were investigated.
    For an arbitraty group $B$ an upper bound of commutator width of $C_p \wr B$ was founded.
%    Also
 %    commutator width  of Sylow 2-subgroups of alternating group ${A_{{2^{k}}}}$, permutation group ${S_{{2^{k}}}}$ are founded. The result of research are extended on subgroups $(Syl_2 {A_{{2^{k}}}})'$, $p>2$.
 %   The structure of commutator subgroup of Sylow 2-subgroups of symmetric and alternating groups is investigated.

    %The paper presents a construction of commutator subgroup of Sylow 2-subgroups of symmetric and alternating groups.
    %Also minimal generic sets of Sylow 2-subgroups of $A_{2^k}$ are founded.

%    Portrait representation of $(Syl_2 {A_{{2^{k}}}})'$, $(Syl_2 {S_{{2^{k}}}})'$ was investigated.

% kpable decision

Key words: wreath product of group, commutator width of wreath product, commutator width of Sylow $p$-subgroups, commutator  and centralizer subgroup of alternating group.\\
\textbf{Mathematics Subject Classification}: 20B27, 20E08, 20B22, 20B35,20F65,20B07, 20E22, 20E45.
  \end{abstract}

% \begin{Title}  \end{Title}

  \begin{section}{Introduction}
  % \note{improve Introduction}

Let G be a group. The commutator width of $G$, $cw(G)$ is defined to be the least
integer $n$, such that every element of $G'$ is a product of at most $n$ commutators
if such an integer exists, and $cw(G) = \infty$ otherwise. The first example of a
finite perfect group with $cw(G) > 1$ was given by Isaacs in \cite{Isacs}.

Commutator width of groups, and of elements has proven to be an important property in particular via its connections with "stable commutator length" and bounded cohomology.

%As it was investigated before by A. Lucchini a case of generating set  of  $C_{p}^{n-1}\wr G$  when $G$  is a finite $n$-generated group and  $p$  is a prime not dividing the order  $|G|$ and $C_p$  be the cyclic group of order $p$ was investigated by A. Lucchini. According to results by A. Lucchini \cite{Luc} the wreath product $C_{p}^{n-1}\wr G$  is also $n$-generated. We consider in role of active group  $G$  the cyclic group, also
%we generalize the passive group of restricted wreath product to any group $B$ instead of direct product of cyclic groups.

%this wreath product on iterated wreath product of such groups and direct product of %wreath products of cyclic groups.
 %The similar question for iterated wreath product was studied by Bondarenko E. V.

A form of commutators of wreath product $A\wr B$ was briefly considered in \cite{Meld} and presented by us as wreath recursion. For more deep description of this form we take into account the commutator width $(cw(G))$ which was presented in work of Muranov  \cite {Mur}.

The form of commutator presentation \cite{Meld} was presented by us in form of wreath recursion \cite{Lav} and commutator width of it was studied. We impose more weak condition on the presentation of wreath product commutator then it was imposed by J. Meldrum.

It was known that, the commutator width of iterated wreath products of nonabelian finite simple groups is bounded by an absolute constant \cite {nikolov, Fite}. But it was not proven that commutator subgroup of $\underset{i=1}{\overset{k}{\mathop{\wr }}}\,{{\mathcal{C}}_{{{p}_{i}}}}$ consists of commutators. We generalize the passive group of this wreath product to any group $B$ instead of only wreath product of cyclic groups and obtain an exact commutator width.
Our goal is to improve these estimations and genralize it on a bigger class of passive groups of wreath product. Also we are going to prove that the commutator width of Sylows $p$-subgroups of symmetric and alternating groups $p \geq 2$ is 1.

% kpable decision
  \end{section}

  \begin{section}{Preliminaries }

We call by ${{g}_{ij}}$ the state of  $g\in Aut{{X}^{[k]}}$ in vertex ${{v}_{ij}}$ as it was called in \cite{Lav, Gr}.
%[Lavrenyuk On the finite state automorphism group of a rooted tree,  % , Olij

Denote by $fun(B,A)$ the direct product of isomorphic copies of A indexed by elements of $B$.
Thus, $fun(B,A)$ is a function $B \rightarrow  A$ with the conventional multiplication and finite supports.
The extension of $fun(B, A)$ by $B$ is called the discrete wreath product of
$A,  B$.
Thus, $A \wr  B: = fun(B, A) \leftthreetimes B$ moreover, $bfb^{-1} = f^b$, $b \in B$, $f \in fun(B, A)$.
As well known that a wreath product of permutation groups is associative construction.

Let $G$ be a group acting (from the left) by permutations on
a set $X$ and let $H$ be an arbitrary group.
Then the (permutational) wreath product
$H \wr G$ is the semidirect product $H^X \leftthreetimes G $, % $H \rightthreetimes G $,
 where $G$ acts on the direct power $H^X$ by
the respective permutations of the direct factors.
The group $C_p$ is equipped with a natural action by the left shift on $X =\{1,…,p\}$, $p\in \mathbb{N}$.
%The multiplication rule of automorphisms $g$, $h$ which presented in form of wreath recursion \cite{Ne}
%$g=(g\rest[1],g\rest[2],\ldots,g\rest[d])\pi_g, \
%h=(h\rest[1],h\rest[2],\ldots,h\rest[d])\pi_h,$ is given by the formula:
%$$g\cdot h=(g\rest[1]h\rest[\pi_g(1)],g\rest[2]h\rest[\pi_g(2)],\ldots,g\rest[d]h\rest[\pi_g(d)])\pi_g\pi_h.$$

The multiplication rule of automorphisms $g$, $h$ which presented in form of the wreath recursion \cite{Ne}
$g=(g\rest[1],g\rest[2],\ldots,g\rest[d])\sigma_g, \
h=(h\rest[1],h\rest[2],\ldots,h\rest[d])\sigma_h,$ is given by the formula:
$$g\cdot h=(g\rest[1]h\rest[\sigma_g(1)],g\rest[2]h\rest[\sigma_g(2)],\ldots,g\rest[d]h\rest[\sigma_g(d)])\sigma_g \sigma_h.$$

We define $\sigma$ as $(1,2,\ldots, p)$ where $p$ is defined by context.
% \note{add wreath recursion description. Product of wreath recursions}
%\note{add wreath product of $C_p$}
%\note{define $\sigma$ as $(1,2,\ldots, p)$ where $p$ is defined by context}
%\noteg{We consider $B\wr (C_p,\, X)$, where $X=\{1,..,p\}$, $B'=\{[f,g]\mid f,g\in B\}$, $p\geq 1$.}

We consider $B \wr (C_p,\, X)$, where $X=\{1,..,p\}$, and $B'=\{[f,g]\mid f,g\in B\}$, $p\geq 1$.
If we fix some indexing $\{x_1, x_2, ... , x_m \}$ of set the $X$, then an element $h\in H^X$ can be written as  $(h_1, ..., h_m ) $ for $h_i \in H$.

 The set $X^*$ is naturally a vertex set of a regular rooted tree, i.e. a connected graph without cycles
and a designated vertex $v_0$ called the root, in which two words are connected by an edge if and only if they are of form $v$ and $vx$, where $v\in X^*$, $x\in X$.
The set $X^n \subset X^*$ is called the $n$-th level of the tree $X^*$
and $X^0 = \{v_0\}$. We denote by $v_{j,i}$ the vertex of $X^j$, which has the number $i$.
Note that the unique vertex $v_{k,i}$ corresponds to the unique word $v$ in alphabet $X$.
For every automorphism $g\in Aut{{X}^{*}}$ and every word $v \in X^{*}$  define the section (state) $g_{(v)} \in AutX^{*}$ of $g$ at $v$ by the rule: $g_{(v)}(x) = y$ for $x, y \in X^*$  if and only if $g(vx) = g(v)y$.
The subtree of $X^{*}$ induced by the set of vertices $\cup_{i=0}^k X^i$ is denoted by $X^{[k]}$.
 The restriction of the action of an automorphism $g\in AutX^*$ to the subtree $X^{[l]}$ is denoted by $g_{(v)}|_{X^{[l]}}$.
 A restriction $g_{(v)}|_{X^{[1]}} $ is called the vertex permutation (v.p.) of $g$ in a vertex $v$.
We call the endomorphism $\alpha|_{v} $ the restriction of $g$ in a vertex $v$ \cite{Ne}. For example, if $|X| = 2$ then we just have to distinguish active vertices, i.e., the
vertices for which $\alpha|_{v} $ is non-trivial.
 As well known if $X=\{0, 1 \} $ then $AutX^{[k-1]} \simeq \underbrace {C_2 \wr ...\wr C_2}_{k-1}$ \cite{Ne}.

Let us label every vertex of ${{X}^{l}},\,\,\,0\le l<k$ by sign 0 or 1 in relation to state of v.p. in it. Let us denote state value of $\alpha $ in $v_{ki}$ as ${{s}_{ki}}(\alpha )$ we put that ${{s}_{ki}}(\alpha )=1$ if $\alpha|_{v_{ki}}  $  is non-trivial, %${{\alpha }_{({{v}_{kl}})}}(x)=y,\,\,\,x\ne y$,
 and $s_{ki}(\alpha )=0$ if $\alpha|_{v_{ki}} $ is trivial.
Obtained by such way a vertex-labeled regular tree is an element of $Aut{{X}^{[k]}}$.
All undeclared terms are from \cite{Sam, Gr}.

Let us make some notations.
 The commutator of two group elements $a$ and $b$, denoted as
$[a,b] = aba\m b\m,$
 conjugation by an element $b$ as
%$[a,b]$ is $aba\m b\m$,x
\[
a^b = bab\m,
\]
$\sigma = (1,2, \ldots, p)$. Also $G_k \simeq Syl_2 A_{2^k}$ that is Sylow 2-subgroup of $A_{2^k}$, $B_k = \wr_{i=1}^k C_2 $. It is convenient do not distinguish $G_k$ $(B_k)$ from its  isomorphic copy in $AutX^{[k]}$. The structure of $G_k$ was investigated in \cite{SkIrred}. For this research we can regard $G_k$ and $B_k$ as recursively constructed i.e.
$B_1 = C_2$,
$B_k = B_{k-1} \wr C_2$   for $k>1$,
$G_1 = \langle e \rangle$,
$G_k = \{(g_1, g_2)\pi \in B_{k} \mid g_1g_2 \in G_{k-1} \}$ for $k>1$.

%$a^b = bab\m$,
%Let $X_1=\{v_{k-1,1}, v_{k-1,2},..., v_{k-1,2^{k-2}} \} $ and $X_2=\{v_{k-1,2^{k-2}+1}, ..., v_{k-1,2^{k-1}} \}$.

The commutator length of an element $g$ of the derived
subgroup of a group $G$ is denoted \emph{clG(g)}, is the minimal $n$ such that there
exist elements $x_1, . . . , x_n, y_1, . . . , y_n$ in G such that $g = [x_1, y_1] . . . [x_n, y_n]$.
The commutator length of the identity element is 0. The commutator width
of a group $G$, denoted $cw(G)$, is the maximum of the commutator lengths
of the elements of its derived subgroup $[G,G]$. It is also related to solvability of quadratic equations in groups \cite{nikolov}.
  \end{section}

  \begin{section}{Main result}

Let $B \wr C_p$ is a regular wreath product of cyclic group of order $p$ and arbitrary group $B$.
%We are going to prove that
%the set of all commutators $K$ of Sylow 2-subgroup $Syl_{2} A_{{2^k}}$ of the alternating  group ${A}_{2^k}$ is the commutator subgroup of $Syl_2 {A_{{2^{k}}}}$.

%\begin{proper}
%\label{G_k_comm_criteria}
%$(g_1, g_2) \in G_k'$ iff $g_1,g_2 \in G_{k-1}$ and $g_1g_2\in F_{k-1}'$.
%\end{proper}

%We will use following theorem which is in \cite{Meld}. Our prove of this theorem is
%more short and we deduce special form of commutator elements.

%We use the following result (one can see this in the book \cite{Meld}).
%The following Lemma follows from Corollary 4.9 of the Meldrum's book \cite{Meld}.

%The following Lemma can be derived with using Corollary 4.9
%of the Meldrum's book \cite{Meld}.
\begin{lemma} \label{form of comm} An element of form
$(r_1, \ldots, r_{p-1}, r_p) \in W'= (B \wr C_p)'$ iff product of all $r_i$ (in any order) belongs to $B'$, where $B$ is an arbitrary group.
\end{lemma} % \cite{Meld}

\begin{proof}
%This Lemma follow the corollary 4.9 of the Meldrum's book \cite{Meld}.
%According to the Corollary 4.9 of the Meldrum's book \cite{Meld} we have the following wreath recursion

Analogously to the Corollary 4.9 of the Meldrum's book \cite{Meld} we can deduce new presentation of commutators in form of wreath recursion
\begin{eqnarray*}
w=(r_1, r_2, \ldots, r_{p-1},  r_p),
\end{eqnarray*}
where $r_i\in B$.
If we multiply elements from a tuple $(r_1, \ldots, r_{p-1}, r_p)$, where $r_i={{h}_{i}}{{g}_{a(i)}}h_{ab(i)}^{-1}g_{ab{{a}^{-1}}(i)}^{-1}$, $h, \, g \in B$ and $a,b \in C_p$, then we get a product
\begin{equation} \label{Meld} %\label{H}
 x=\stackrel{p}{ \underset{\text{\it i=1}} \prod} r_i = \prod\limits_{i=1}^{p}{{{h}_{i}}{{g}_{a(i)}}h_{ab(i)}^{-1}g_{ab{{a}^{-1}}(i)}^{-1} \in B'},
\end{equation}
where $x $ is a product of corespondent commutators.
Therefore, we can write $r_p = r_{p-1}\m \ldots r_1\m x$. We can rewrite element $x\in B'$ as the product $x = \prod \limits ^{m}_{j=1} [f_j,g_j]$,  $m \le cw(B)$.

%we impose more strong condition
Note that we impose more weak condition on the product of all $r_i$ to belongs to $B'$ then in Definition 4.5. of form $P(L)$ in \cite{Meld}, where the product of all $r_i$ belongs to a subgroup $L$ of $B$ such that $ L>B'$.

 In more detail deducing of our representation constructing can be reported in following way.
 If we multiply elements having form of a tuple $(r_1, \ldots, r_{p-1}, r_p)$, where $r_i={{h}_{i}}{{g}_{a(i)}}h_{ab(i)}^{-1}g_{ab{{a}^{-1}}(i)}^{-1}$, $h, \, g \in B$ and $a,b \in C_p$, then in case $cw(B)=0$ we obtain a product
\begin{equation}\label{Meld2}
 \stackrel{p}{ \underset{\text{\it i=1}} \prod} r_i = \prod\limits_{i=1}^{p}{{{h}_{i}}{{g}_{a(i)}}h_{ab(i)}^{-1}g_{ab{{a}^{-1}}(i)}^{-1} \in B'}.
\end{equation}

Note that if we rearange elements in (1) as $h_{1} h_{1}^{-1} g_{1}g_2^{-1}h_{2} h_{2}^{-1} g_{1}g_2^{-1} ...  h_{p} h_{p}^{-1} g_{p}g_p^{-1}$ then by the reason of such permutations we obtain a product of corespondent commutators. %$\prod\limits_{i=1}^{p-1}(h_{i} h_{i}^{-1} g_{i}g_i^{-1})x \in B'$,
%$\prod\limits_{i=1}^{p-1}(h_{i} h_{i}^{-1} g_{i}g_i^{-1})x \in B'$
% = a_ib_j^{-1}a_{m}b_{l}^{-1}.
Therefore, following equality holds true

\begin{equation}\label{HH}
\prod\limits_{i=1}^{p}{{{h}_{i}}{{g}_{a(i)}}h_{ab(i)}^{-1}g_{ab{{a}^{-1}}(i)}^{-1} } =\prod\limits_{i=1}^{p}h_{i} h_{i}^{-1} g_{i}g_i^{-1}x \in B',
\end{equation}
where $x $ is a product of corespondent commutators.
Therefore,
\begin{eqnarray} \label{form}
(r_1, \ldots, r_{p-1}, r_p) \in W' \mbox{ iff } r_{p-1} \cdot \ldots \cdot r_{1} \cdot r_p = x\in B'.
\end{eqnarray}
 Thus, one element from states of wreath recursion $(r_1, \ldots, r_{p-1}, r_p) $ depends on rest of $r_i$. This dependence contribute that the product $\prod\limits_{j=1}^{p}r_{j}$ for an arbitrary sequence $\{ r_{j} \}_{j=1} ^{p}$
%$r_{p-1} \cdot \ldots \cdot r_{1} \cdot r_p $
 belongs to  $B'$. Thus, $r_p$ can be expressed as:
\begin{eqnarray*}
r_p = r_{1}\m \cdot \ldots \cdot r_{p-1}\m x.
\end{eqnarray*}

Denote a $j$-th tuple, which consists of a wreath recursion elements, by $(r_{{j}_1},r_{{j}_2},..., r_{{j}_p} )$.
Closedness by multiplication of the set of forms $(r_1, \ldots, r_{p-1}, r_p) \in W= (B \wr C_p)'$
% of elements from a set $K$ of all commutators
  follows from
  %$\prod\limits_{j=1}^{k}  ( r_{j_1} \ldots r_{j_{p-1}} r_{j_p})= \prod\limits_{j=1}^{k} \prod\limits_{i=1}^{p}  r_{j_i} =\prod\limits_{i=1}^{p}  r_{j_i} \prod\limits_{i=1}^{p}  r_{i_2} ... \prod\limits_{i=1}^{p}  r_{i_k} = R_1 R_2 ...  R_{k} \in B '$,

  %                        VERSION OF RECORD 2

\begin{eqnarray} \label{prod}
  \prod\limits_{j=1}^{k}  ( r_{j1} \ldots r_{j{p-1}} r_{jp})= \prod\limits_{j=1}^{k} \prod\limits_{i=1}^{p}  r_{j_i} =  R_1 R_2 ...  R_{k} \in B ',
\end{eqnarray}

  where $r_{ji}$ is $i$-th element from the tuple number $j$,  $R_j = \prod\limits_{i=1}^{p}  r_{ji}, \,\, \, 1 \leq j \leq  k$. As it was shown above $R_j = \prod\limits_{i=1}^{p-1}  r_{ji} \in B'$. Therefore, the product (\ref{prod}) of $R_j$, $j \in \{1,...,k \}$ which is similar to the product mentioned in \cite{Meld}, has the property $R_1 R_2 ...  R_{k} \in B '$ too, because of $B '$ is subgroup.
   Thus, we get a product of form (\ref{Meld}) and the similar reasoning as above are applicable.
%   $r_{p-1} \cdot \ldots \cdot r_{1} \cdot r_p$
% transitivity of action $C_p$ on $X$ an

Let us prove the sufficiency condition. % Wise versa
If the set $K$ of elements satisfying the condition of this theorem, that all products of all $r_i$, where every $i$ occurs in this forms once, belong to $B'$, then using the elements of form

 $(r_{1},e,..., e, r_{1}^{-1} )$, ... , $(e,e,...,e, r_{i}, e, r_{i}^{-1} )$, ... ,$(e,e,..., e, r_{p-1}, r_{p-1}^{-1})$, $(e,e,..., e, r_1 r_2 \cdot ...\cdot r_{p-1} )$

  we can express any element of form $(r_1, \ldots, r_{p-1}, r_p) \in W= (C_p \wr B)'$. We need to prove that in such way we can express all element from $W$ and only elements of $W$. The fact that all elements can be generated by elements of $K$ follows from randomness of choice every $r_i$, $i<p$ and the fact that equality (1) holds so construction of $r_p$ is determined.
%Hence, following equality holds true
%\begin{equation}\label{HH}
%\prod\limits_{i=1}^{p}{{{h}_{i}}{{g}_{a(i)}}h_{ab(i)}^{-1}g_{ab{{a}^{-1}}(i)}^{-1} } =\prod\limits_{i=1}^{p}h_{i} h_{i}^{-1} g_{i}g_i^{-1}x \in B',
%\end{equation}
\end{proof}
%where $r_i\in B$ and $r_1 r_2 \ldots r_{p-1} r_p = x \in B'$. Therefore we can write $r_p = r_{p-1}\m \ldots r_1\m x$. We also can rewrite element $x\in B'$ as product of commutators $x = \prod \limits ^{cw(B)}_{j=1} [f_j,g_j]$.

\begin{lemma} \label{form of comm_2} For any group $B$ and integer $p\geq 2, \, p \in \mathbb{N} $ if $w\in (B \wr C_p)'$ then $w$ can be represented as the following wreath recursion
\begin{align*}
w=(r_1, r_2, \ldots, r_{p-1},  r_1\m \ldots r_{p-1}\m \prod \limits ^{k}_{j=1} [f_j,g_j]),
\end{align*}
where $r_1, \ldots, r_{p-1}, f_j, g_j \in B$, and $k\leq cw(B)$.
\end{lemma}
\begin{proof}
%This Lemma follow the corollary 4.9 of the Meldrum's book \cite{Meld}.
%According to the corollary 4.9 of the Meldrum's book \cite{Meld} we have the following wreath recursion
According the Lemma~\ref{form of comm} we have the following wreath recursion
\begin{align*}
w=(r_1, r_2, \ldots, r_{p-1},  r_p),
\end{align*}
where $r_i\in B$ and $r_{p-1} r_{p-2} \ldots r_2 r_1  r_p = x \in B'$. Therefore, we can write $r_p = r_1\m \ldots r_{p-1}\m x$. We also can rewrite element $x\in B'$ as product of commutators $x = \prod \limits ^{k}_{j=1} [f_j,g_j]$, where $k\leq cw(B)$.
%Analogously to the Corollary 4.9 of the Meldrum's book \cite{Meld} we can make new presentation of commutators in form of wreath recursion
%\begin{align*}
%w=(r_1, r_2, \ldots, r_{p-1},  r_p),
%\end{align*}
%where $r_i\in B$ and $r_{p-1} r_{p-2} \ldots r_2 r_1  r_p = x \in B'$. Therefore we can write $r_p = r_1\m \ldots r_{p-1}\m x$. We also can rewrite element $x\in B'$ as product of commutators $x = \prod \limits ^{cw(B)}_{j=1} [f_j,g_j]$.
\end{proof}

\begin{lemma} \label{c_p_wr_b_elem_repr}
For any group $B$ and integer $p\geq 2, \, p \in \mathbb{N}$ if $w\in (B \wr C_p)'$ is defined by the following wreath recursion
\begin{align*}
w=(r_1, r_2, \ldots, r_{p-1},  r_1\m \ldots r_{p-1}\m [f,g]),
\end{align*}
where $r_1, \ldots, r_{p-1}, f_j, g_j \in B$, then $w$ can be represent  as commutator
\begin{align*}
w = [(a_{1,1},\ldots, a_{1,p})\sigma, (a_{2,1},\ldots, a_{2,p})],
\end{align*}
where
\begin{align*}
a_{1,i} &=  e \mbox{ for $1\leq i \leq p-1$ },\\
a_{2,1} &= (f\m)^{r_1\m \ldots r_{p-1}\m},\\
a_{2,i} &= r_{i-1} a_{2,i-1}\mbox{ for $2\leq i \leq p$},\\
a_{1,p} &= g^{a_{2,p}\m}.
\end{align*}
\end{lemma}
\begin{proof}
%We consider the following elements
%\begin{align*}
%(a_{1,1},\ldots, a_{1,p})\sigma \in C_p\wr B \mbox{ and } (a_{2,1},\ldots, a_{2,p}) \in C_p\wr B.
%\end{align*}
%We consider the commutator of these elements
Let us consider the following commutator
\begin{align*}
\kappa &= (a_{1,1},\ldots, a_{1,p})\sigma \cdot (a_{2,1},\ldots, a_{2,p}) \cdot (a_{1,p}\m,a_{1,1}\m,\ldots, a_{1,p-1}\m)\sigma\m \cdot (a_{2,1}\m,\ldots, a_{2,p}\m)\\
&= (a_{3,1}, \ldots, a_{3,p}),
%=(\ldots, a_{1,i}a_{2,1 + (i \bmod p)}a_{1,i}\m a_{2,i}\m, \ldots)
\end{align*}
where
\begin{align*}
a_{3,i} = a_{1,i}a_{2,1 + (i \bmod p)}a_{1,i}\m a_{2,i}\m.
\end{align*}
At first we compute the following
\begin{align*}
a_{3,i} = a_{1,i}a_{2,i+1}a_{1,i}\m a_{2,i}\m = a_{2,i+1} a_{2,i}\m = r_{i} a_{2,i} a_{2,i}\m=  r_i, \mbox{ for $1\leq i \leq p-1$}.
\end{align*}
Then we make some transformations of $a_{3,p}$:
\begin{align*}
a_{3,p}&=a_{1,p}a_{2,1}a_{1,p}\m a_{2,p}\m\\
&=(a_{2,1} a_{2,1}\m) a_{1,p}a_{2,1}a_{1,p}\m a_{2,p}\m\\
&=a_{2,1} [a_{2,1}\m, a_{1,p}] a_{2,p}\m\\
&=a_{2,1}a_{2,p}\m a_{2,p} [a_{2,1}\m, a_{1,p}] a_{2,p}\m\\
%&=(a_{2,p} a_{2,1}\m)\m a_{2,p} [a_{2,1}\m, a_{1,p}] a_{2,p}\m\\
&= (a_{2,p} a_{2,1}\m)\m  [(a_{2,1}\m)^{a_{2,p}}, a_{1,p}^{a_{2,p}}]\\
&= (a_{2,p} a_{2,1}\m)\m  [(a_{2,1}\m)^{a_{2,p} a_{2,1}\m}, a_{1,p}^{a_{2,p}}].
\end{align*}
%We transform commutator $\kappa$ in such form that is similar to the form of $w$. This gives us equations with unknown variables $a_{i,j}$:
We transform the commutator $\kappa$ to a form analogous to that of $w$:
\begin{eqnarray*}
\left\{
\begin{matrix}
a_{1,i}a_{2,i+1}a_{1,i}\m a_{2,i}\m &=& r_i, \mbox{ for $1\leq i \leq p-1$},\\
(a_{2,p} a_{2,1}\m)\m &=& r_1\m \ldots r_{p-1}\m,\\
(a_{2,1}\m)^{a_{2,p} a_{2,1}\m} &=& f,\\
a_{1,p}^{a_{2,p}} &=& g.
\end{matrix}\right.
\end{eqnarray*}
In order to prove required statement it is sufficient to find at least one solution of equations. We set the following
\begin{eqnarray*}
a_{1,i} &=  e \mbox{ for $1\leq i \leq p-1$ }.
\end{eqnarray*}
Then we have
\begin{eqnarray*}
\left\{
\begin{matrix}
a_{2,i+1} a_{2,i}\m &=& r_i, \mbox{ for $1\leq i \leq p-1$},\\
(a_{2,p} a_{2,1}\m)\m &=& r_1\m \ldots r_{p-1}\m,\\
(a_{2,1}\m)^{a_{2,p} a_{2,1}\m} &=& f,\\
a_{1,p}^{a_{2,p}} &=& g.
\end{matrix}\right.
\end{eqnarray*}
Now we can see that the form of the commutator $\kappa$ is analogous to the form of $w$.

Let us make the following notation
\begin{align*}
%r' = r_1\m \ldots r_{p-1}\m.
r' = r_{p-1} \ldots r_1.
\end{align*}
We note that from the definition of $a_{2, i}$ for $2\leq i \leq p$ it follows that
\begin{align*}
r_i = a_{2,i+1} a_{2,i}\m, \mbox{ for $1\leq i \leq p-1$}.
\end{align*}
Therefore,
\begin{align*}
r' &=  (a_{2,p} a_{2,p-1}\m) (a_{2,p-1} a_{2,p-2}\m)\ldots (a_{2,3} a_{2,2}\m) (a_{2,2} a_{2,1}\m)\\
&= a_{2,p} a_{2,1}\m.
\end{align*}
And then
\[
(a_{2,p} a_{2,1}\m)\m = (r')\m = r_1\m \ldots r_{p-1}\m.
\]
Finally let us to compute the following
\begin{align*}
%a_{3,i} = a_{1,i}a_{2,i+1}a_{1,i}\m a_{2,i}\m = a_{2,i+1} a_{2,i}\m = r_{i} a_{2,i} a_{2,i}\m=  r_i, \mbox{ for $1\leq i \leq p-1$},\\
%(a_{2,p} a_{2,1}\m)\m = r_1\m \ldots r_{p-1}\m,\\
(a_{2,1}\m)^{a_{2,p} a_{2,1}\m} = (((f\m)^{r_1\m \ldots r_{p-1}\m})\m)^{r'} = (f^{(r')\m})^{r'}  = f,\\
a_{1,p}^{a_{2,p}} = (g^{a_{2,p}\m})^{a_{2,p}} = g.
\end{align*}
And now we conclude that
\begin{align*}
a_{3,p} = r_1\m \ldots r_{p-1}\m [f,g].
\end{align*}
Thus, the commutator $\kappa$ is presented exactly in the similar form as $w$ has.
\end{proof}
For future using we formulate previous lemma for the case $p=2$.

\begin{corollary} \label{c_2_wr_b_elem_repr}
If $B$ is any group and $w\in (B \wr C_2)'$ is defined by the following wreath recursion
\begin{align*}
w=(r_1,  r_1\m [f,g]),
\end{align*}
where $r_1, f, g \in B$, then $w$ can be represent as commutator
\begin{align*}
w = [(e, a_{1,2})\sigma, (a_{2,1}, a_{2,2})],
\end{align*}
where
\begin{align*}
a_{2,1} &= (f\m)^{r_1\m},\\
a_{2,2} &= r_{1} a_{2,1},\\
a_{1,2} &= g^{a_{2,2}\m}.
\end{align*}
\end{corollary}
The proof is immediate from  Lemma \ref{c_p_wr_b_elem_repr}  in the case where $p=2$.

\begin{lemma} \label{comm B_k}
For any group $B$ and integer $p\geq 2$ inequality
\begin{align*}
cw(B\wr C_p) \leq \max(1,cw(B))
\end{align*}
holds.
\end{lemma}
\begin{proof}
We can represent any $w\in (B \wr C_p)'$ by Lemma~\ref{form of comm} as the following wreath recursion
\begin{align*}
w&=(r_1, r_2, \ldots, r_{p-1},  r_1\m \ldots r_{p-1}\m  \prod \limits_{j=1}^{k} [f_{j},g_{j}])\\
 &= (r_1, r_2, \ldots, r_{p-1},  r_1\m \ldots r_{p-1}\m  [f_{1},g_{1}]) \cdot \prod \limits_{j=2}^{k} [(e, \ldots, e, f_j), (e, \ldots, e, g_j)],
\end{align*}
where $r_1, \ldots, r_{p-1}, f_j, g_j \in B$, $k \leq cw(B)$. Now we can apply Lemma~\ref{c_p_wr_b_elem_repr} to the element  \newline
$(r_1, r_2, \ldots, r_{p-1},  r_1\m \ldots r_{p-1}\m  [f_{1},g_{1}])$. Lemma~\ref{c_p_wr_b_elem_repr} implies that $w$ can be represented as product of $\max(1, cw(B))$ commutators.
\end{proof}
%\note{todo}

\begin{corollary} \label{comm Cycl}
If  $W = C_{p_k} \wr \ldots \wr C_{p_1}$ then
$cw(W) =1$ for $k\geq 2$.
%\note{wreath product of cyclic groups and inverse limits}
\end{corollary}
\begin{proof}
If $B= C_{p_k} \wr C_{p_{k-1}}$  then taking into consideration that $cw(B)>0$ (because $C_{p_k} \wr C_{p_{k-1}}$ is not commutative group). Since Lemma \ref{comm B_k} implies that $cw(C_{p_k} \wr C_{p_{k-1}})=1$ then according to the inequality $cw(C_{p_k} \wr C_{p_{k-1}} \wr C_{p_{k-2}}) \leq \max(1,cw(B))$ from Lemma \ref{comm B_k} we obtain $cw(C_{p_k} \wr C_{p_{k-1}} \wr C_{p_{k-2}})=1$. Analogously if $W = C_{p_k} \wr \ldots \wr C_{p_1}$ and supposition of induction for $C_{p_{k}} \wr \ldots \wr C_{p_2}$ holds, then using an associativity of a permutational wreath product we obtain from the inequality of Lemma \ref{comm B_k} and the equality $cw( C_{p_k} \wr \ldots \wr C_{p_2})=1$ that $cw(W)=1$.
\end{proof}

We define our partial ordered set $M$ as the set of all finite wreath products of cyclic groups. We make of use directed set $\mathbb{N}$.
\begin{equation} \label{cwH }
{{H}_{k}}=\underset{i=1}{\overset{k}{\mathop{\wr }}}\,{{\mathcal{C}}_{{{p}_{i}}}}
\end{equation}

Moreover, it has already been proved in Corollary  \ref{comm Cycl} that each group  of the form $\underset{i=1}{\overset{k}{\mathop{\wr }}}\,{{\mathcal{C}}_{{{p}_{i}}}}$  has a  commutator  width equal to 1, i.e $cw(\underset{i=1}{\overset{k}{\mathop{\wr }}}\,{{\mathcal{C}}_{{{p}_{i}}}})=1$.   A partial order  relation will be a subgroup relationship.  Define the injective homomorphism $f_{k,k+1}$ from the $\underset{i=1}{\overset{k}{\mathop{\wr}}}\,{{\mathcal{C}}_{{{p}_{i}}}}$ into $\underset{i=1}{\overset{k+1}{\mathop{\wr }}}\,{{\mathcal{C}}_{{{p}_{i}}}}$ by mapping a generator of active group ${\mathcal{C}_{{{p}_{i}}}}$ of ${{H}_{k}}$ in a generator of active group ${\mathcal{C}_{{{p}_{i}}}}$ of ${{H}_{k+1}}$.
%, further the generator of ${{C}_{{{p}_{i}}}}$ in generator of $i$-th  subgroup of ${{H}_{k+1}}$ that is ${{C}_{{{p}_{i}}}}$.
In more details the injective homomorphism $f_{k,k+1}$ is defined as $g \mapsto g(e,..., e)$, where a generator $g\in \underset{i=1}{\overset{k} {\mathop{\wr }}}\, {\mathcal{C}}_{{p}_{i}}$, $g(e,..., e)\in  \underset{i=1}{\overset{k+1}{\mathop{\wr }}}\, {\mathcal{C}}_{{p}_{i}}$.

Therefore this is an injective homomorphism of ${{H}_{k}}$ onto subgroup $\underset{i=1}{\overset{k}{\mathop{\wr }}}\,{{\mathcal{C}}_{{{p}_{i}}}}$ of $H_{k+1}$.

\begin{corollary} \label{comm Cycl}
The direct limit  $\underrightarrow{\lim }\underset{i=1}{\overset{k}{\mathop{\wr }}}\,{{\mathcal{C}}_{{{p}_{i}}}}$ of  direct system $\left\langle {{f}_{k,j}},\,\underset{i=1}{\overset{k}{\mathop{\wr }}}\,{{\mathcal{C}}_{{{p}_{i}}}} \right\rangle $ has commutator width 1.
\end{corollary}

\begin{proof}
We make the transition to the direct limit in the directed system $\left\langle {{f}_{k,j}},\,\underset{i=1}{\overset{k}{\mathop{\wr }}}\,{{\mathcal{C}}_{{{p}_{i}}}} \right\rangle $  of injective mappings from chain $e\to \,\,...\,\,\to \underset{i=1}{\overset{k}{\mathop{\wr }}}\,{{\mathcal{C}}_{{{p}_{i}}}}\to \underset{i=1}{\overset{k+1}{\mathop{\wr }}}\,{{\mathcal{C}}_{{{p}_{i}}}}\to \underset{i=1}{\overset{k+2}{\mathop{\wr }}}\,{{\mathcal{C}}_{{{p}_{i}}}}\to ...$.

Since all mappings in chains were injective homomorphisms, it have a trivial kernel, so the transition to a direct limit boundary preserves the property $\rho$: that any element of commutator subgroup is commutator, because  in each group ${{H}_{k}}$ from the chain endowed by $\rho $.
%Therefore for the direct limit $\underrightarrow{\lim }\underset{i=1}{\overset{k}{\mathop{\wr }}}\,{{\mathcal{C}}_{{{p}^{i}}}}$ property $\rho $ holds.

The direct limit of the direct system  is denoted by $\underrightarrow{\lim }\underset{i=1}{\overset{k}{\mathop{\wr }}}\,{{\mathcal{C}}_{{{p}_{i}}}}$ and is defined as disjoint union of the  ${{H}_{k}}$'s modulo a certain equivalence relation:

$$\underrightarrow{\lim }\underset{i=1}{\overset{k}{\mathop{\wr }}}\,{{\mathcal{C}}_{{{p}_{i}}}}={}^{\coprod\limits_{k}{\underset{i=1}{\overset{k}{\mathop{\wr }}}\,{{\mathcal{C}}_{{{p}_{i}}}}}}/{}_{\sim }.$$

Since every element $g$ of $\underrightarrow{\lim }\underset{i=1}{\overset{k}{\mathop{\wr }}}\,{{\mathcal{C}}_{{{p}_{i}}}}$ coincides with some element from one of the groups ${{G}_{m}}$ of directed system, then by the injectivity of the mappings for $g$
the property  $cw(\underset{i=1}{\overset{k}{\mathop{\wr }}}\,{{\mathcal{C}}_{{{p}_{i}}}})=1$ also holds. Thus, it holds for the whole $\underrightarrow{\lim }\underset{i=1}{\overset{k}{\mathop{\wr }}}\,{{\mathcal{C}}_{{{p}_{i}}}}$.

\end{proof}

%\begin{lemma} \label{X^2} Restrictions of states $g_{2i}\in Aut X^*$ having index $i={1,2}$ are equal also $g_{23}=g_{24}$.
%\end{lemma}

%\begin{lemma} \label{X^2} The following congruences of  states restrictions $g_{21}=g_{22}$, $g_{23}=g_{24}$ are true. \end{lemma}
%\begin{lemma} \label{X^2} The following congruences of  states restrictions hold: $g_{21}=g_{22}$, $g_{23}=g_{24}$.
%\end{lemma}

\begin{corollary} \label{cw_syl_p_s_p_k_eq_1_and_syl_p_a_p_k_eq_1}  Commutator width  $cw(Syl_p(S_{p^k})) = 1$ for prime $p$ and $k>1$ and commutator width $cw(Syl_p (A_{p^k})) = 1$ for prime $p>2$ and $k>1$.
\end{corollary}
\begin{proof}
Since $Syl_p(S_{p^k}) \simeq  \stackrel{k}{ \underset{\text{\it i=1}}{\wr }}C_p$  (see \cite{Kal, Paw, Sk} %\note{is it correct?})
 then $cw(Syl_p(S_{p^k}))=1$. As well known $Syl_p S_{p^k} \simeq Syl_p A_{p^k} $ in case $p>2$  (see \cite{Dm}, then $cw(Syl_p(A_{p^k}))=1$ if $p>2$.
\end{proof}

The following Lemma gives us a criteria of belonging an element from the group $Syl_2 S_{2^k}$ to $(Syl_2 S_{2^k})'$. %This method based on checking the presentation by portraits of automorphism of binary restricted tree. In detail, our algorithm check the indexes of mentioned automorphism on parity.

%Recall that $  \stackrel{k}{ \underset{\text{\it i=1}}{\wr }}C_p  \simeq Syl_p S_{p^k} $.
\begin{lemma} \label{comm B_k old} An element $g \in B_k$ belongs to commutator subgroup $B'_k$ iff $g$ has an even index on ${{X}^{l}}$ for all $0\leq l < k$. Where $B_k \simeq Syl_2 S_{2^k}$.
%Commutators of all elements from $B_k$ acting on ${{X}^{[k]}}$
% have all possible even indexes
%  on ${{X}^{l}},\,\,l<k$.
%The set of all commutators $K$ of Sylow 2-subgroup $Syl_{2} A_{{2^k}}$ of the alternating  group ${A}_{2^k}$ is the commutant of $Syl_2 {A_{{2^{k}}}}$.
%A commutator of elements from $Sy{{l}_{2}}{{A}_{{{2}^{k}}}}$ is an element
%  possessing all possible even indexes value
% on ${{X}^{l}},\,\,l<k-1$ and all possible even indexes on ${{X}^{k-1}}$ in intersection with %${{v}_{11}}{{X}^{[k-1]}}$ and with  ${{v}_{12}}{{X}^{[k-1]}}$.
 \end{lemma}

\begin{proof}
Let us prove the necessity by induction by a number of level $l$ and index of $g$ on $X^l$.
 We first show that our statement for base of the induction is true. Actually, if $\alpha, \beta \in B_{0}$ then $(\alpha \beta {{\alpha }^{-1}}) \beta^{-1}$ determines a trivial v.p. on $X^{0}$. If $\alpha, \beta\in B_{1}$ and $\beta$ has an odd index on $X^1$, then $(\alpha \beta {{\alpha }^{-1}})$ and $\beta^{-1}$ have the same index on $X^1$. Consequently, in this case an index of the product $(\alpha \beta {{\alpha }^{-1}}) \beta^{-1}$ can be 0 or 2. Case where $\alpha, \beta \in B_{1}$ and has even index on $X^1$, needs no proof, because the product and the sum of even numbers is an even number.

To finish the proof it suffices to assume that for $B_{l-1}$ statement holds and prove that it holds for $B_l$.
Let $\alpha, \beta$ are an arbitrary automorphisms from $Aut X^{[k]}$ and $\beta$ has index $x$ on $X^{l}, \, l<k$, where $0 \leq x \leq 2^{l} $.% , then conjugation $(\alpha \beta {{\alpha }^{-1}})$ acts by arbitrary permutation of active vertices that has $\beta$ on $X^l$.
   A conjugation of an automorphism $\beta $  by arbitrary $\alpha \in Aut{{X}^{[k]}}$ gives us arbitrary permutations of $X^l$ where $\beta $ has active v.p.

 Thus following product $(\alpha \beta {{\alpha }^{-1}}) \beta ^{-1}$ admits all possible even indexes on $X^l,  l<k$ from 0 to $2x$.
 %In addition $[\alpha, \beta]$ can has an arbitrary permitted assignment (arrangement) of v.p. on $X^l$.

  Let us present $B_k$ as $B_k=B_l \wr B_{k-l}$, so elements $\alpha, \beta$ can be presented in form of wreath recursion $\alpha = (h_{1},...,h_{2^l })\pi_1, \,  \beta = (f_{1},...,f_{2^l })\pi_2$, $h_{i}, f_{i} \in B_{k-l} ,\  0<i \leq 2^l$ and $h_{i}, f_{j}$ corresponds to sections of automorphism in vertices of $X^{l+1}$. % of isomorphic subgroup to $B_{k-l}$ in $Aut X^{[k]}$.

% Множення автоморфізмів $g$ і $h$, записаних у вигляді
%$$g=(g\rest[1],g\rest[2],\ldots,g\rest[d])\pi_g, \
%h=(h\rest[1],h\rest[2],\ldots,h\rest[d])\pi_h,$$ визначається за
%правилом:
%$$g\cdot h=(g\rest[1]h\rest[\pi_g(1)],g\rest[2]h\rest[\pi_g(2)],\ldots,g\rest[d]h\rest[\pi_g(d)])\pi_g\pi_h.$$
%$$g\cdot h=(g\rest[1]h\rest[\pi_g(1)],g\rest[2]h\rest[\pi_g(2)],\ldots,g\rest[d]h\rest[\pi_g(d)])\pi_g\pi_h.$$
Actually, the parity of this index are formed independently of the action of
$Aut X^{[l]}$ on $X^l$. So this index forms as a result of multiplying of elements of commutator presented as wreath recursion $(\alpha \beta \alpha^{-1}) \cdot \beta ^{-1} = (h_{1},...,h_{2^l})\pi_1 \cdot (f_{1},...,f_{2^l })\pi_2= (h_{1},...,h_{2^l}) (f_{\pi_1(1)},...,f_{\pi_1(2^l)})\pi_1 \pi_2 $, where $h_{i}, f_{j} \in {B}_{k-l}$, $l<k$.
Let there are $x$ automorphisms, that have active vertex at the level $X^l$, among $h_i$.
Analogous automorphisms $h_{i}$ has number of active v.p. equal to $x$. As a result of multiplication we have automorphism with index $2i:$ $0 \leq 2i \leq 2x$ on $X^l$.
%$h_{i1}, h_{i2} \in St_{B_k} (l)$
Consequently, commutator $[\alpha, \beta]$ has an arbitrary even indexes on $X^m$, $m<l$ and we showed by induction  that it has even index on $X^l$.

Let us prove this Lemma by induction on level $k$. Let us to suppose that we prove current Lemma (both sufficiency and necessity) for $k-1$. Then we rewrite element $g \in B_k$ with wreath recursion
\[
g = (g_1, g_2)\sigma^i,
\]
where $i \in \{0, 1\}$.

Now  we consider sufficiency.

Let $g \in B_k$ and $g$ has all even indexes on $X^j$ $0 \leq j < k$ we need to show that $g \in B'_k$. According to condition of this Lemma $g_1 g_2$ has even indexes. An element $g$ has form $g=(g_1, g_2)$, where $g_1, g_2 \in B_{k-1}$, and products $g_1 g_2 = h \in B'_{k-1}$ because $h \in B_{k-1}$ and for $B_{k-1}$ induction assumption holds. Therefore, all products of form $g_1 g_2$ indicated in formula \ref{Meld} belongs to $B'_{k-1}$. Hence, from Lemma \ref{form of comm} follows that $g= (g_1, g_2 ) \in B'_k$.
% assumption of mathematical induction for $g_1, g_2 \in B_{k-1}$ holds viz $g_1, g_2 \in B'_{k-1}$. sufficiency
\end{proof}

An automorphisms group of the subgroup ${C_{2}^{{{2}^{k-1}-1}}}$ is based on permutations of copies of $C_2$. Orders of $\stackrel{k-1}{ \underset{\text{\it i=1}}{\wr } }C_2  $ and ${C_{2}^{{{2}^{k-1}-1}}}$ are equals. A homomorphism from $\stackrel{k-1}{ \underset{\text{\it i=1}}{\wr } }C_2  $ into $Aut({C_{2}^{{{2}^{k-1}-1}}})$ is injective because a kernel of action $\stackrel{k-1}{ \underset{\text{\it i=1}}{\wr } }C_2  $ on ${C_{2}^{{{2}^{k-1}-1}}}$ is trivial, action is effective. The group $G_k$ is a proper subgroup of index 2 in the group $\stackrel{k}{ \underset{\text{\it i=1}}{\wr } }C_2  $ \cite{Dm, SkIrred, Sk }.
%ПРАВКИ
The following theorem can be used for proving structural property of Sylow subgroups.
\begin{theorem} \label{max}
\textbf{A maximal 2-subgroup}  of $Aut{{X}^{\left[ k \right]}}$ acting by even permutations on ${{X}^{k}}$ has the structure of the semidirect product $G_k \simeq  B_{k-1} \ltimes W_{k-1} $ and is isomorphic to $Syl_2A_{2^k}$. Also $G_k < B_k$.
\end{theorem}
%ПРАВКИ
This theorem is proven by the author in \cite{SkIrred}.
\begin{proposition}\label{B_k_criteria}
An element $(g_1, g_2)\sigma^i$, $i \in \{0, 1\} $ of wreath power $\stackrel{k}{ \underset{\text{\it i=1}}{\wr } }C_2  $ belongs to its subgroup $ G_k$, iff $g_1g_2 \in G_{k-1}$.
%$(g_1, g_2)\sigma^i \in G_k$
\end{proposition}
\begin{proof}
%content...
This fact follows from the structure of elements of $G_k$ described Theorem \ref{max} in and the construction of wreath recursion. Indeed, due to the structure of elements of $G_k$ described in Theorem \ref{max}, we have an action on $X^k$ by an even permutations because the subgroup $W_{k-1}$, containing an even number of transposition, acts on $X^k$ only by even permutation. The condition $g_1g_2 \in G_{k-1}$ on states $g_1, g_2$ of automorphism $g=(g_1, g_2) \sigma^{i} $  is equivalent to conditions that index of $g$ on $X^{k-1}$ is even but this condition equivalent to the condition that $g$ acting on $X^k$ by even permutation.
\end{proof}

\begin{lemma} \label{comm_rel}
The subgroup  $G_{k}^{'}$ has an even index of any level of  ${{X}^{\left[ k \right]}}$.
\end{lemma}

According to Lemma  \ref{comm B_k old} and  Property  \ref{B_k_criteria} a number of active vertices index of any level of  ${{X}^{\left[ k \right]}}$  is even. Therefore, equality
\begin{equation} \label{prod_l}
\prod\limits_{j=1}^{{{2}^{i-1}}}{{{g}_{ij}}}=\prod\limits_{j={{2}^{i-1}}+1}^{{{2}^{i}}}{{{g}_{ij}}}		
\end{equation}

is true. In case $i=k-1$ there are the even number of transpositions on each of last levels of ${{v}_{11}}{{X}^{\left[ k-2 \right]}}$  and of  ${{v}_{12}}{{X}^{\left[ k-2 \right]}}$. Therefore the following condition

\begin{equation} \label{prod_k}
\prod\limits_{j=1}^{{{2}^{k-2}}}{{{g}_{k-1j}}}=\prod\limits_{j={{2}^{k-2}}+1}^{{{2}^{k-1}}}{{{g}_{k-1j}}}=e	
\end{equation}
holds.

\begin{corollary}
 The rooted permutation in ${{v}_{0}}$ of any automorphism from ${{G}_{k}}'$  is always trivial also ${{g}_{11}}={{g}_{12}}$.
\end{corollary}
\begin{proof}
The proof is immediate follows from Lemma \ref{comm_rel} in particular  ${{g}_{11}}={{g}_{12}}$ immediately follows from the equality  (\ref{prod_l}). The triviality of v.p. in ${{v}_{0}}$ implies from the fact of level order permutation parity in сommutant which implies from parity of level index.
\end{proof}

Recall that a subgroup $A$ of direct product is called subdirect product of groups ${{H}_{i}}$ if projection of $A$ on any subgroup ${{H}_{i}}$ coincides with the  subgroup ${{H}_{i}}$ \cite{Kar}.
We denote by $\bot $ operation of even subdirect product which  admits all  possible tuples of active vertices  from multipliers of form ${{G}_{k-1}}$ with condition of parity of transpostions number in the product ${{G}_{k-1}}\bot {{G}_{k-1}}$ on any level ${{X}^{l}},\,l>k-1$.

As it was said above the subgroup ${{G}_{k-1}}$ is isomorphic copy of $Sy{{l}_{2}}{{A}_{{{2}^{k-1}}}}$ in the group $Aut{{X}^{[k-1]}}$.

\begin{theorem} \label{comm_str} The group ${{G}_{k}}'$  has a structure ${{G}_{k}}'\simeq {{G}_{k-1}}~\bot {{G}_{k-1}}$.
\end{theorem}

\begin{proof}
Relation (\ref{prod_l}) from Lemma \ref{comm_rel} implies the parity of the level index $I{{n}_{l}}(g)$,
	Conversely the operation $\bot $ admits all  possible tuples of transpositions from multipliers of form ${{G}_{k-1}}$ with condition of parity of transpostions number in the product ${{G}_{k-1}}\bot \,{{G}_{k-1}}$ on any level  ${{X}^{l}},\,l>k-1$. Since parity of transpositions in both factors is the same and so condition (\ref{prod_l}) is satisfied.
\end{proof}

Let us denote by ${{s}_{ij}}$  state of automorphism $G''$ in vertex ${{v}_{ij}}$. Recall that any automorphism of  $g\in Aut{{X}^{[n]}}$  can be uniquely represented as  $g=\left( {{g}_{x}},\,\,x\in X \right)\pi $ \cite{Lav}.  We find the  commutator $\left[ g,\,h \right]$ in the form of the wreath recursion $({{g}_{21}},\,{{g}_{22}},{{g}_{23}},\,{{g}_{24}})\pi $, where $\pi $ is a rooted permutation of first level states.

The portraits of automorphisms $\theta$ on level $X^{k-1}$ can be characterized by the sequence $(s_1, s_2, ... , s_{2^{k-1}})$, $s_i \in {0, 1}$.
The set of vertices of $X^{k-1}$ can be disjoint into subsets $X_1$ and $X_2$.
Let $X_1=\{v_{k-1,1}, v_{k-1,2},..., v_{k-1,2^{k-2}} \} $ and $X_2=\{v_{k-1,2^{k-2}+1}, ..., v_{k-1,2^{k-1}} \}$.

We call a distance structure $\rho_l(\theta)$ of $\theta $ a tuple of distances between its active vertices from $X^l$.
Let group $Sy{{l}_{2}}{{A}_{{{2}^{k}}}}$ acts on $X^{[k]}$.

\begin{corollary} \label{states second comm}  If $g,h\in G'$ then  states of $s=[g,\,\,h]\in G''\left| _{{{X}^{\left[ 3 \right]}}} \right.$ in vertices of ${{X}^{2}}$  comply with the equalities \[{{s}_{21}}={{s}_{22}},\,\,\,\,{{s}_{23}}={{s}_{24}}. \]
\end{corollary}

\begin{proof}
Since if $\alpha $  is a vertex permutation of $g$ an ${{v}_{2j}}$, then we do not distinguish $\alpha $ from the section ${{g}_{2j}}$ defined by it,  i.e., we can write ${{g}_{2j}}=\alpha $. Note that permutations of $g\in {{G}^{''}}$ in vertices of ${{X}^{0}}$ and ${{X}^{1}}$are trivial, conversely   ${{g}_{21}}$ is coincide with v.p.  in ${{v}_{21}}$.
The proof is based on  Lemma  \ref{comm_rel}. Taking  in the consideration that ${{G}_{k}}'\simeq {{H}_{k-1}}\bot {{H}_{k-1}}$ it is  enough to prove that ${{s}_{21}}={{s}_{22}}$ the prove  of ${{s}_{23}}={{s}_{24}}$  is analogous.   Therefore we consider projection ${{g}_{1}},\,\,{{h}_{1}}$  of  $g,h\in G'$ on ${{H}_{k-1}}$  such that  ${{g}_{1}}\in {{H}_{k-1}}\triangleleft Aut{{v}_{11}}{{X}^{[k-1]}}$,  ${{h}_{1}}\in {{H}_{k-1}}\triangleleft Aut{{v}_{11}}{{X}^{[k-1]}}$. Thus it can be presented in form   ${{g}_{1}}=({{g}_{21}},\,{{g}_{22}})\sigma ,\,\,{{h}_{1}}=({{h}_{21}},\,{{h}_{22}})\pi \in G'$, where  $\sigma ,\,\,\pi \in {{S}_{2}}$.
 We find the commutator $[{{g}_{1}},\,\,{{h}_{1}}]$ in form of wreath recursion $\left( {{s}_{21}},\,\,{{s}_{22}} \right){{\pi }_{1}}$.
$g=\left( {{g}_{21}},{{g}_{22}},{{g}_{23}},{{g}_{24}} \right)\sigma =\left( \left( {{g}_{21}},{{g}_{22}}){{\sigma }_{1}},({{g}_{23}},{{g}_{24}} \right){{\sigma }_{2}} \right)$, where ${{g}_{2i}},\,\,\,1\le i\le 4$ are states of second level, $\sigma $ are rooted permutation of states of first level, because v.p. above are trivial, which permutes vertices of second level with its subtrees and ${{\sigma }_{1}},\,\,{{\sigma }_{2}}$ rooted permutations of  subtrees with roots in ${{v}_{21}},\,\,{{v}_{22}}$ and ${{v}_{23}},\,\,{{v}_{24}}$. Therefore $\sigma =\left( {{\sigma }_{1}},\,{{\sigma }_{2}} \right)$.
 We shall consider $g=({{g}_{21}},\,{{g}_{22}})\sigma ,\,\,h=({{h}_{21}},\,{{h}_{22}})\pi \in G'$, where \[\sigma =(1,\,2),\,\,\,\,\pi =e. \] Therefore $g\in {{H}_{k-1}}\triangleleft Aut{{v}_{11}}{{X}^{[k-1]}}$, $h\in {{H}_{k-1}}\triangleleft Aut{{v}_{11}}{{X}^{[k-1]}}$. This case is possible because we take 2 different elements $g,\,\,h\in G'$. Then we have  $[({{g}_{21}},{{g}_{22}})\sigma ,\,\,\,\,({{h}_{21}},{{h}_{22}})\pi ]=({{g}_{21}}{{h}_{22}}{{g}_{21}}{{h}_{21}},\,{{g}_{22}}{{h}_{21}}{{g}_{22}}{{h}_{22}})$ we conclude that it is
sufficient to prove equality ${{s}_{21}}={{s}_{22}}$. The prove that  ${{s}_{23}}={{s}_{24}}$ is the same.  Elements from first factor ${{H}_{k-1}}$ of   ${{G}_{k}}'$  has form $g=({{g}_{21}},{{g}_{22}})\sigma ,\,\,\,h=({{h}_{21}},{{h}_{22}})\pi \in {{G}_{k-1}}\triangleleft Aut{{v}_{12}}{{X}^{[k-1]}}$. To find $[g,h]$ in rest of cases besides $\sigma =(1,\,2),\,\,\,\,\pi =e, $  we have to shortly consider the feasible  cases:

1. ~ \ \ \ \ \ \ \ \ \ \  \ \ \ \ \  \ \ \ \ \ \ \ \
$\sigma ={{\overset{\scriptscriptstyle\frown}{g}}_{11}}=e,\,\,\,\,\,\pi ={{\overset{\scriptscriptstyle\frown}{g}}_{12}}=e,$ \\
Therefore we obtain a commutator
 $[({{g}_{21}},{{g}_{22}})\sigma ,\,\,({{h}_{21}},{{h}_{22}})\pi ]=({{g}_{21}}{{h}_{21}}{{g}_{21}}{{h}_{21}},\,{{g}_{22}}{{h}_{22}}{{g}_{22}}{{h}_{22}}).$

2.  ~ \ \ \ \ \ \ \ \ \ \  \ \ \ \ \  \ \ \ \ \ \ \ \  $\sigma =\,\pi ={{\overset{\scriptscriptstyle\frown}{g}}_{11}}={{\overset{\scriptscriptstyle\frown}{g}}_{12}}=(1,2),$

 then $[({{g}_{21}},{{g}_{22}})\sigma ,\,\,\,\,({{h}_{21}},{{h}_{22}})\pi ]=({{g}_{21}}{{h}_{22}}{{g}_{22}}{{h}_{21}},\,\,{{g}_{22}}{{h}_{21}}{{g}_{21}}{{h}_{22}}).$
In view of the fact of commutativity of ${{g}_{21}},\,\,{{g}_{22}},\,\,\,{{g}_{23}},\,\,{{g}_{24}}\in {{S}_{2}}$ we have \[{{g}_{21}}{{h}_{22}}{{g}_{22}}{{h}_{21}}=\,\,{{g}_{22}}{{h}_{21}}{{g}_{21}}{{h}_{22}},\] viz ${{s}_{21}}={{s}_{22}}$.

3. Case \[\sigma =e,\,\,\,\,\,\pi =(1,\,2).\] Then we have \[[({{g}_{21}},{{g}_{22}})\sigma ,\,\,\,\,({{h}_{21}},{{h}_{22}})\pi ]=({{g}_{21}}{{h}_{21}}{{g}_{22}}{{h}_{21}},\,{{g}_{22}}{{h}_{22}}{{g}_{21}}{{h}_{22}}).\] Consider  states of  this commutator \[{{g}_{21}}{{h}_{21}}{{g}_{22}}{{h}_{21}}={{g}_{21}}{{g}_{22}}{{h}_{21}}{{h}_{21}}={{g}_{21}}{{g}_{22}}=\,{{g}_{22}}{{g}_{21}}{{h}_{22}}{{h}_{22}}=\,{{g}_{22}}{{h}_{22}}{{g}_{21}}{{h}_{22}}.\] Thus, we see that the states are equal again, i.e. ${{s}_{21}}={{s}_{22}}$.

4. And recall the  case \[\sigma =(1,\,2),\,\,\,\,\pi =e\] is completely analogous to case 3 because\[[({{g}_{21}},{{g}_{22}})\sigma ,\,\,\,\,({{h}_{21}},{{h}_{22}})\pi ]=({{g}_{21}}{{h}_{22}}{{g}_{21}}{{h}_{21}},\,{{g}_{22}}{{h}_{21}}{{g}_{22}}{{h}_{22}}).\]
thus ${{s}_{21}}={{s}_{22}}.$

%We have to show that  the equalities  \[{{s}_{21}}={{s}_{22}},\,\,\,\,{{s}_{23}}={{s}_{24}}\] hold.

  Consider  the product of  commutators  $\left[ g,\,h \right]\cdot \left[ {{g}_{1}},\,{{h}_{1}} \right]$ it already proven that
      $I{{n}_{2}}(g)$, $I{{n}_{2}}(h)$ and $I{{n}_{2}}({{g}_{1}})$, $I{{n}_{2}}({{h}_{1}})$ are even its states satisfy  ${{s}_{21}}={{s}_{22}},\,\,\,\,{{s}_{23}}={{s}_{24}}$. Recall that  elements of $G''$ have form $\left( {{s}_{21}},\,\,{{s}_{22}} \right){{\pi }_{1}}$, $\left( {{s}_{23}},\,\,{{s}_{24}} \right){{\pi }_{2}}$, where ${{\pi }_{1}}$, ${{\pi }_{2}}$  are trivial  then elements $\left[ g,\,h \right],\,\,\,\left[ {{g}_{1}},\,{{h}_{1}} \right]$ multiply directly. Consequently, the equality of states of product $\left[ g,\,h \right]\cdot \left[ {{g}_{1}},\,{{h}_{1}} \right]$ holds.
\end{proof}

\begin{lemma} \label{comm}
An element $g$ belongs to $G_k' \simeq Syl_2{A_{2^k}}$ iff $g$ is an arbitrary element from $G_k$ which has all even indexes  on ${{X}^{l}},\,\,l<k-1$ of ${{X}^{[k]}}$ and on ${{X}^{k-2}}$ of subtrees ${{v}_{11}}{{X}^{[k-1]}}$ and ${{v}_{12}}{{X}^{[k-1]}}$.

%Commutators of all elements from $Sy{{l}_{2}}{{A}_{{{2}^{k}}}}$
% have all possible even indexes
%  on ${{X}^{l}},\,\,l<k-1$ of ${{X}^{[k]}}$ and on ${{X}^{k-2}}$ of subtrees ${{v}_{11}}{{X}^{[k-1]}}$ and ${{v}_{12}}{{X}^{[k-1]}}$, $(Sy{{l}_{2}}{{A}_{{{2}^{k}}}})'$ consists of commutators.

%The set of all commutators $K$ of Sylow 2-subgroup $Syl_{2} A_{{2^k}}$ of the alternating  group ${A}_{2^k}$ is the commutant of $Syl_2 {A_{{2^{k}}}}$.

%A commutator of elements from $Sy{{l}_{2}}{{A}_{{{2}^{k}}}}$ is an element
%  possessing all possible even indexes value
% on ${{X}^{l}},\,\,l<k-1$ and all possible even indexes on ${{X}^{k-1}}$ in intersection with %${{v}_{11}}{{X}^{[k-1]}}$ and with  ${{v}_{12}}{{X}^{[k-1]}}$.
 \end{lemma}

\begin{proof}
%As we said any automorphism $\theta$ from $G_k$ has even index on $X_{k-1}$ so $\theta$ has the same parity of numbers of active v.p. on $X_1$ and $X_2$.
%Let us consider ${\alpha }\in Aut{{v}_{11}}{{X}^{\left[ k-1 \right]}}$, conjugation by such ${\alpha }$ of any $\theta \in G_k$ only permute vertices inside $X_1$ ($X_2$).

Let us prove the ampleness by induction on a number of level $l$ and index of automorphism $g$ on $X^l$.
%Recall that any automorphism $\theta \in Syl_2 A_n$ has an even index on $X^{k-1}$ so number parities of active v. p. on ${{X}_{1}}$ and on ${{X}_{2}}$ are the same.
Conjugation by automorphism $\alpha= \alpha_0$ from $Aut{{v}_{11}}{{X}^{\left[ k-1 \right]}}$ of automorphism $\theta $, that has index $x:$ $1 \leq x \leq 2^{k-2}$ on ${{X}_{1}}$ does not change $x$. Also automorphism $\theta^{-1} $ has the same number $x$ of v. p. on $X_{k-1}$ as $\theta $ has. If $\alpha$ from $Aut{{v}_{11}}{{X}^{\left[ k-1 \right]}}$ and $ \alpha \notin v_{12} Aut{{X}^{\left[ k \right]}}$ then conjugation $(\alpha \theta {{\alpha }^{-1}})$ permutes on $X^{k-1}$ vertices which of $X_1$.

 Thus, ${\alpha }\theta {\alpha^{-1} }$ and $\theta$ have the same parities of number of active v.p. on $X_1$ ($X_2$). Hence, a product ${\alpha }\theta {\alpha^{-1} } \theta^{-1}$ has an even number of active v.p. on $X_1$ ($X_2$) in this case. More over  a coordinate-wise sum by \texttt{mod2} %$mod2$
  of active v. p. from $(\alpha \theta {{\alpha }^{-1}})$ and $\theta^{-1}$ on $X_1$ ($X_2$) is even and equal to $y:$ $0 \leq y \leq 2x$.

%Because number parities of active v.p. from $(\alpha \theta {{\alpha }^{-1}})$ and ${{\theta }^{-1}}$ from $X_1$ ($X_2$) are the same.

   If conjugation by $\alpha$ permutes sets $X_1$ and $X_2$ then there are  coordinate-wise sums of no trivial v.p. from $\alpha \theta \alpha^{-1} \theta^{-1}$ on $X_1$ (analogously on $X_2$) have form: \\ $({{s}_{k-1,1}}(\alpha \theta {{\alpha }^{-1}}),..., {{s}_{k-1, 2^{k-2}}}(\alpha \theta {{\alpha }^{-1}}) )\oplus ({{s}_ {k-1,1}}(\theta^{-1}), ..., {{s}_{k-1,{{2}^{k-2}}}}(\theta^{-1} ))$.
  % ,..., {{s}_{{k-1,}}{2^{k-2}+1}}(\alpha \theta {{\alpha }^{-1}}
   This sum has even number of v.p. on $X_1$ and $X_2$ because $(\alpha \theta {{\alpha }^{-1}})$ and ${{\theta }^{-1}}$ have a same parity of no trivial v.p. on $X_1$ ($X_2$).  Hence, $(\alpha \theta {{\alpha }^{-1}}){{\theta }^{-1}}$ has even number of v.p. on ${{X}_{1}}$ as well as on ${{X}_{2}}$.

%On $X^l, \, l< k-1$ commutator $(\alpha \theta {{\alpha }^{-1}}){{\theta }^{-1}}$ has even index because if $\theta $ has odd (even) index then $\theta^{-1}$ and $\alpha \theta {{\alpha }^{-1}}$ have odd (even) indexes so coordinate-wise sum by $mod2$  of active v.p. from $\alpha \theta {{\alpha }^{-1}}$ and ${{\theta }^{-1}}$ is always even.

An automorphism $\theta $ from $G_k$ was arbitrary so number  of active v.p. $x$ on $X_1$ is an arbitrary $0 \leq x\leq 2^l$. And ${\alpha }$ is and arbitrary from $AutX^{[k-1]}$ so vertices can be permuted in such way that the commutator $[{\alpha },\theta]$ has arbitrary even number $y$ of active v.p. on $X_1$, $0 \leq y \leq 2x$.

 A conjugation of an automorphism $\theta $ having index $x$, $1 \leq x \leq 2^{l}$ on ${{X}^{l}}$  by different $\alpha \in Aut{{X}^{[k]}}$  gives us all tuples of active v.p. with the same $\rho_l(\theta)$ that $\theta $ has on ${{X}^{l}}$, by which $Aut{{X}^{[k]}}$ acts on $X^l$. Let supposition  of induction for element $g$  with index $2k-2$ on $X^l$ holds so $g=(\alpha \theta {{\alpha }^{-1}}){{\theta }^{-1}}$, where $In_l(\theta)=x$. To make a induction step we complete $\theta$ by such vertex  permutation in $v_{l,x}$ too $\theta$ has suitable distance structure for $g=(\alpha \theta {{\alpha }^{-1}}){{\theta }^{-1}}$, also if $g$ has rather different distance structure $\rho_l(g)$ from $\rho_l(\theta)$ then have to change $\theta$. In case when we complete $\theta$ by $v_{l,x}$ it has too satisfy a condition $(\alpha \theta {{\alpha }^{-1}}) (v_{x+1})=v_{l, y}$, where $v_{l, y}$ is a new active vertex of $g$ on $X^l$.
  Note that $v(x+1) $ always can be chosen such that acts in such way $\alpha(v(x+1)) = v(2k+2)$ because action  of $\alpha$ is 1-transitive. Second vertex arise when we multiply $(\alpha \theta {{\alpha }^{-1}})$ on $\theta^{-1}$. %Thus $in_l(g)$ became to be $2k+2$.
 %  we complete $\theta$ by one active vertex on $X^l$ let this vertex $v_x+1$
 % and $(\alpha \theta {{\alpha }^{-1}}) (v_{x+1})=v_{2k+1}$
  %Note that $\alpha $ always can be chosen such that acts in such way $\alpha(v(x+1)) = v(2k+2)$ because action is 1-transitive.
  Hence $In_l(\alpha \theta {{\alpha }^{-1}})=2k+2$ and coordinates of new vertices $v_{2k+1}, v_{2k+2}$ are arbitrary from 1 to $2^l$.
%  Note that $\alpha $ always can be chosen such that acts non-trivial on $v(2k-1)$.

%  Choice of these vertices is such too $(\alpha \theta {{\alpha }^{-1}})(v_{x+1})=v_{2k+1}$ and $(\alpha \theta {{\alpha }^{-1}})(v_{x+2})=v_{2k+2}$, where $\alpha$ is the same. Hence $In_l(\alpha \theta {{\alpha }^{-1}}))=2k+2$.

 So multiplication $(\alpha \theta {{\alpha }^{-1}})\theta $ generates a commutator having index $y$  equal to coordinate-wise sum by $mod 2$ of no trivial v.p. from vectors $({{s}_{l1}}(\alpha \theta {{\alpha }^{-1}}),{{s}_{l}}_{2}(\alpha \theta {{\alpha }^{-1}}),...,{{s}_{l{{2}^{l}}}}(\alpha \theta {{\alpha }^{-1}}))\oplus ({{s}_{l1}}(\theta ),{{s}_{l}}_{2}(\theta ),...,{{s}_{l{{2}^{l}}}}(\theta ))$  on ${{X}^{l}}$. A indexes parities of  $\alpha \theta {{\alpha }^{-1}}$  and  ${{\theta }^{-1}}$ are same so their sum by $mod 2$ are even.  Choosing $\theta $ we can  choose an arbitrary index $x$ of $\theta $ also we can choose arbitrary $\alpha $ to make a permutation of active v.p. on ${{X}^{l}}$.  Thus, we obtain an element with arbitrary even index on ${{X}^{l}}$ and arbitrary location of active v.p. on ${{X}^{l}}$.

%Зміни
Check that property of number parity of v.p. on ${{X}_{1}}$  and on ${{X}_{2}}$  is closed with respect to conjugation. We know that numbers of active v. p. on ${{X}_{1}}$ as well as on ${{X}_{2}}$ have the same parities. So
action by conjugation only can permutes it, hence, we again get the same  structure of element. Conjugation by automorphism $\alpha $ from  $Aut{{v}_{11}}{{X}^{\left[ k-1 \right]}}$  automorphism $\theta $, that has odd number of  active v. p. on ${{X}_{1}}$  does not change its parity.
Choosing the $\theta $ we can choose arbitrary index $x$ of
$\theta $ on ${{X}^{k-1}}$ and number of active v.p. on ${{X}_{1}}$  and  ${{X}_{2}}$  also we can choose arbitrary $\alpha $ to make a permutation active v.p. on ${{X}_{1}}$  and  ${{X}_{2}}$. Thus, we can generate all possible elements from a commutant. Also this result can be deduced due to Lemma \ref{comm B_k old}.

Let $\kappa_1, \kappa_2 \in K$ and each of which has even index on $X^l$ and $2^l$-tuples of v.p. $({{s}_{l,1}}(\kappa_1),..., {{s}_{k-1, 2^l}}(\kappa_1) )$, $({{s}_ {l,\kappa_1(1)}}(\kappa_2), ..., {{s}_{l,\kappa_1({{2}^{l}})}}(\kappa_2 ))$ corresponds to portrait of $\kappa_1$, $\kappa_2$ on $X^l$.
  Then a number of non-trivial coordinates in a coordinate-wise sum \\
   $({{s}_{l,1}}(\kappa_1),..., {{s}_{k-1, 2^l}}(\kappa_1) )\oplus ({{s}_ {l,\kappa_1(1)}}(\kappa_2), ..., {{s}_{l,\kappa_1({{2}^{l}})}}(\kappa_2 ))$ is even.

Let us check that the set of all commutators $K$ from $Syl_2 A_{2^k}$ is closed with respect  to multiplication of commutators.
%Let $\kappa_1, \kappa_2 \in K$, then $\kappa_1 \kappa_2$ has an even index on $X^l$, $l<k-1$ because  coordinate-wise sum $({{s}_{l,1}}(\kappa_1),..., {{s}_{k-1, 2^l}}(\kappa_1) )\oplus ({{s}_ {l,\kappa_1(1)}}(\kappa_2), ..., {{s}_{l,\kappa_1({{2}^{l}})}}(\kappa_2 ))$ contains even number of non-trivial elements.
% of two $2^l$-tuples of v.p. with an even number of no trivial coordinate has even number of such coordinate.
  Note that conjugation of $\kappa $ can permute sets ${{X}_{1}}$ and ${{X}_{2}}$  so parities of $x_1$ and $X_2$ coincide. It is obviously that the parity of index of $\alpha \kappa \alpha^{-1}$ is the same as index of $\kappa $.

Check that a set $K$ is a set  closed with respect  to conjugation.

 Let $\kappa \in K$, then $\alpha \kappa {{\alpha }^{-1}}$  also belongs to $K$, it is so because  conjugation does not change index of an automorphism on a level. Conjugation only  permutes vertices on a level because elements of $Aut{{X}^{\left[ l-1 \right]}}$ acts  on vertices of  ${{X}^{l}}$. But as it was proved above elements  of $K$ have all possible indexes on ${{X}^{l}}$, so as a result of conjugation $\alpha \kappa {{\alpha }^{-1}}$ we obtain an element from $K$.

Check that the set of commutators is closed with respect to multiplication of commutators.
Let $\kappa_1, \kappa_2 $ be an arbitrary commutators of $G_k$. The parity of the number of vertex permutations on $X^l$ in the product $\kappa_1 \kappa_2 $  is determined exceptionally by the parity of the numbers of active v.p. on ${{X}^{l}}$ in $\kappa_1$ and $\kappa_2$ (independently from the action of v.p. from the higher levels). Thus $\kappa_1 \kappa_2 $ has an even index on $X^l$.

 Hence, a normal closure of the set $K$ coincides with $K$. It means that commutator subgroup of $Syl_2 A_{2^k}$ consists of commutators.
\end{proof}

\begin{proposition}\label{L_k_comm_criteria}
An element $(g_1, g_2)\sigma^i \in G_k'$ iff $g_1,g_2 \in G_{k-1}$ and $g_1g_2\in B_{k-1}'$.
\end{proposition}

\begin{proof}
 Since, if $(g_1, g_2) \in G_k'$ then indexes of $g_1$ and $g_2$ on $X^{k-1}$ are even according to Lemma \ref{comm} thus, $g_1,g_2 \in G_{k-1}$. A sum of indexes of $g_1$ and $g_2$ on $X^{l}$, $l<k-1$ are even according to Lemma \ref{comm} too, so index of product $g_1 g_2$ on $X^{l}$ is even. Thus, $g_1g_2\in B_{k-1}'$. Hence, necessity is proved.
%Hence, necessity for $(g_1, g_2) \in G_k'$ is proved.

Let us prove the sufficiency via Lemma \ref{comm}.
Wise versa, if $g_1,g_2 \in G_{k-1}$ then indexes of these automorphisms on $X^{k-2}$ of subtrees $v_{11}X^{[k-1]}$ and $v_{12}X^{[k-1]}$ are even as elements from $ G_k'$ have.  %Automorphism $g_1$ is restriction of automorphism of $X^{[k]}$ on subtree $v_{11}X^{[k]}$.
 The product $g_1g_2$ belongs to $B_{k-1}'$ by condition of this Lemma and so sum of indexes of $g_1, g_2$ on any level $X^l$,  $0 \leq l<k-1 $ is even. Thus, the characteristic property of $G_k'$ described in Lemma \ref{comm} holds.
\end{proof}

\begin{lemma} \label{comm1}
%Commutators  of all elements from $G_k$
% have all possible even indexes
%  on ${{X}^{l}},\,\,l<k-1$ of ${{X}^{[k]}}$ and on ${{X}^{k-2}}$ of subtrees ${{v}_{11}}{{X}^{[k-1]}}$ and %${{v}_{12}}{{X}^{[k-1]}}$.
%Formulation 2.
An element $g=(g_1, g_2)\sigma^{i}$ of $G_k$, $i\in\{0,1\}$ belongs to $G'_k$ iff $g$ has even index on ${{X}^{l}}$ for all $l<k-1$ and elements $g_1, g_2$ have even indexes on ${{X}^{k-1}}$, that is equally matched to $g_1, g_2 \in G_{k-1}$.
 \end{lemma}
 \begin{proof}
The proof immediately follows from Lemma \ref{comm} and Proposition \ref{L_k_comm_criteria}.
%\note{add proof}
 \end{proof}

Recall the Lemma on the commutators structure for an embedded commutator \cite{Gural}.
%We recall the Statement on the structure of commutators for an embedded commutator.

\begin{lemma}\label{Gural} Suppose $G$ is a  group with a subgroup $H$ such that $H \rhd G'$
and $G = <H, x>$. If  $w$ is a commutator in $G$, then $w = [axe, b]$  for some
$a,\text{ }b\in H$ and  $e\in Z$.
%Понятно, что $x\in G\text{ }\!\!\backslash\!\!\text{ }\,H$.
\end{lemma}

If we assume that ${{B}_{k}}=G$, $H={{G}_{k}}$ and since  ${{G}_{k}}\supseteq B_{k}^{'}$, then according to Lemma \ref{Gural} any element  $w\in B_{k}^{'}$ can be presented as commutator of an element from ${{B}_{k}}$ and an extending element.
In our case as an extending element for maximal subgroup  $H={{G}_{k}}$ to ${{B}_{k}}$ we could take the generator ${{\alpha }_{k-1}}$ \cite{SkIrred}. Thus, any element $w\in B_{k}^{'}$ can be presented as $w = [axe, b]$.  Where $x$ is an extending element for ${{G}_{k}}$ to ${{B}_{k}}$.
%In particular, it is convenient to choose $x={{\alpha }_{k-1}}$ where this element is as described in \cite{SkSyl}.
Also as an extending element it can be chosen an arbitrary element with an odd index on $k-1$ level and zero indexes on rest of levels.

\begin{proposition}\label{comm_F_k_is_subgroup_of_L_k}
The following inclusion $B_k'<G_k$ holds.
\end{proposition}
 \begin{proof}
 Since  $B_k' = \wr_{i=1}^{k-1} C_2 = B_{k-1}$ and $G_k \simeq  B_{k-1} \ltimes W_{k-1} $ we have $B_k' < G_k$.
 \end{proof}
\begin{proposition}\label{G_k_is_normal_in_B_k}
The group $G_k$ is normal in wreath product $\stackrel{k}{ \underset{\text{\it i=1}}{\wr } }C_2$ i.e. $ G_k \lhd B_k$.
\end{proposition}
\begin{proof}
The commutator of $B_k$ is $B_k' < B_{k-1}$. In other hand $ B_{k-1} <  G_{k}$ because $G_k \simeq  B_{k-1} \ltimes W_{k-1} $ consequently $B_k' < G_{k}$. Thus, $G_{k} \lhd B_{k}$.
\end{proof}
There exists a normal embedding  (normal  injective monomorphism)  $\varphi :\,\,{{G}_{k}}\to {{B}_{k}}$  \cite{Heinek} i.e.$~~{{G}_{k}}\triangleleft {{B}_{k}}$. Actually, it implies from  Proposition \ref{G_k_is_normal_in_B_k}. Also according to \cite{SkIrred} the index $\left| {{B}_{k}}~ :~{{G}_{k}} \right|=2$ so ${G}_k$ is a normal subgroup that is a factor subgroup $~~{}^{{{B}_{k}}}/{}_{{{C}_{2}}}\simeq {{G}_{k}}$.
%  Фразу such that $~~{{G}_{k}}\triangleleft {{B}_{k}}$ лутше в статье закоментироавть ведь это растолкование термина normal embedding.

\begin{theorem}\label{_comm_F_k_eq_[L_k,F_k]}
Elements of $\aut[k]'$ have the following form $\aut[k]'=\{[f,l]\mid f\in B_k, l\in G_k\}=\{[l,f]\mid f\in B_k, l\in G_k\}$.
\end{theorem}
\begin{proof}
It is enough to show either $B_k'=\{[f,l]\mid f\in B_k, l\in G_k\}$ or $B_k'=\{[l,f]\mid f\in B_k, l\in G_k\}$ because if $f = [g,h]$ then $f\m = [h,g]$.

We prove this Theorem by induction on $k$.
Since $B_1' = \langle e \rangle$ then base of induction is verified.

Due Lemma \ref{form of comm_2}  we already know that every element $w\in\aut[k]'$ can be represent as
\begin{align*}
w=(r_1, r_1\m [f,g])
\end{align*}
for some $r_1,f\in \aut[k-1]$ and $ g\in \syl[k-1]$ (by induction hypothesis). By the Corollary~\ref{c_2_wr_b_elem_repr} we can represent $w$ as commutator of
\begin{align*}
(e,a_{1,2})\sigma \in \aut[k] \mbox{ and } (a_{2,1}, a_{2,2}) \in \aut[k],
\end{align*}
where
\begin{align*}
a_{2,1} &= (f\m)^{r_1\m},\\
a_{2,2} &= r_{1} a_{2,1},\\
a_{1,2} &= g^{a_{2,2}\m}.
\end{align*}
We note that $g \in G_{k-1}$ then by Proposition~\ref{B_k_criteria} we obtain $(e,a_{1,2})\sigma \in \syl[k]$.
\end{proof}
Directly from this Proposition follows next Remark, that needs no proof.
\begin{remark}\label{_comm_B_k_eq_[b_k,b_k]}
Let us to note that Theorem~\ref{_comm_F_k_eq_[L_k,F_k]} improve Corollary~\ref{cw_syl_p_s_p_k_eq_1_and_syl_p_a_p_k_eq_1} for the case $p=2$.
%$\syl[k]'=\{[f_1,f_2]\mid f_1\in \syl[k], f_2\in \syl[k]\}$.
%The set $B'_k$ coincides with set of all commutators of $B_k$, put it differently $cw(B_k)=1$.
\end{remark}

%\begin{proposition}
%For all $g\in F_k$: $g^2\in F_k'$
%\end{proposition}
%\begin{proof}
%Induction on $k$: $F_1=C_2$
%\begin{align*}
%g = (g_1, g_2) \sigma^i\\
%g^2 = (g_1^2, g_2^2)\mbox{ or } g^2 = (g_1g_2, g_2g_1)
%\end{align*}
%\begin{align*}
%g_1^2 \in F_{k-1}'\\
%g_2^2 \in F_{k-1}'\\
%g_1^2 g_2^2 \in F_{k-1}'
%\end{align*}
%Then $g^2 = (g_1^2, g_2^2) \in F_{k}'$.
%\begin{align*}
%g_1 g_2^2 g_1 = g_1^2 g_2^2 [g_2^{-2}, g_1\m] \in F_{k-1}'
%\end{align*}
%\end{proof}

\begin{proposition}\label{B'_k and B^2_k}
If $g $ is an element of wreath power $\stackrel{k}{ \underset{\text{\it i=1}}{\wr } }C_2 \simeq B_k $ then $g^2 \in B'_{k}$.
\end{proposition}
\begin{proof}
As it was proved in Lemma \ref{comm B_k old} commutator $[\alpha, \beta]$ from $B_k$ has arbitrary even indexes on $X^m$, $m<k$. Let us show that elements of $B_k^2$ have the same structure.
%Conversely, commutator $[\alpha, \beta]$ has arbitrary even indexes on $X^m$, $m<l$ by assumption of an induction and we showed that it has even index on $X^l$.

Let $\alpha, \beta \in B_k$ an indexes of the automorphisms $\alpha^2 $, $(\alpha \beta)^2 $  on $X^l, \, l<k-1$ are always even. In more detail the indexes of $\alpha^2 $, $(\alpha \beta)^2 $ and $\alpha^{-2} $ on $X^l$ are determined exceptionally by the parity of indexes of $\alpha $ and $\beta $ on ${{X}^{l}}$. Actually, the parity of this index are formed independently of the action of
$Aut X^l$ on $X^l$.
So this index forms as a result of multiplying of elements $\alpha \in B_k$ presented as wreath recursion $ \alpha^2 = (h_{1},...,h_{2^l})\pi_1 \cdot (h_{1},...,h_{2^l })\pi_1= (h_{1},...,h_{2^l}) (h_{\pi_1(1)},...,h_{\pi_1(2^l)})\pi_1 ^2 $, where $h_{i}, h_{j} \in {B}_{k-l}, \, \pi_1 \in B_l$, $l<k$.
 Let there are $x$ automorphisms, that have an active vertex at level $X^l$, among $h_i$.
% And besides automorphisms corresponding to $h_{i}$ are $x$ automorphisms which has active v.p. on $X^l$.
Analogous automorphisms $h_{i}$ has number of active v.p. equal to $x$. As a result of multiplication we have automorphism with index $2i:$ $0 \leq 2i \leq 2x$.
%of v.p. from the higher levels and this index is even.
%Since an index of $\alpha \beta $ on ${{X}^{l}}$ is an arbitrary $x:$ $0\leq x \leq 2^l$ then an index of $(\alpha \beta)^2 $ is arbitrary even number that is between $0$ and $2^l$.

Since $g^2 $ admits only an even index on $X^l$ of $Aut X^{[k]}$, $0<l<k$, then $g^2 \in B'_{k}$ according to Lemma \ref{comm B_k old}. %about structure of a commutator subgroup
\end{proof}

Since as well known a group $G_k^2$ contains the subgroup $G'$ then a product $G^2 G'$ contains all elements from the commutant. Therefore, we obtain that $G_k^2 \simeq G'_k$.

%Alternative prove for $G'_k$ is as follows. Here is a NEW our PROVE !!!
\begin{proposition}\label{g_sq_in_G_k}
For arbitrary $g\in G_k$ following inclusion $g^2\in G_k'$ holds.
\end{proposition}
\begin{proof}
We make the prove by the induction on positive integer $k$. Elements of
 $G^2_1$ have form $ ((e, e)\sigma)^2 = e $, where $\sigma = (1,2)$ so statement holds. In general case when $k>1$ elements of $G_k$ have the following form
%Induction on $k$: element of
\begin{align*}
g = (g_1, g_2) \sigma^i,  \, g_1 \in B_{k-1}, \, i \in \{0.1\}. \\
\mbox{ Hence we have two possibilities} \\
g^2 = (g_1^2, g_2^2)\mbox{ or } g^2 = (g_1g_2, g_2g_1).
\end{align*}

We first show that
\begin{align*}
g_1^2 \in B_{k-1}',
g_2^2 \in B_{k-1}'\\
\mbox{ after we will prove} \\
g_1 g_2 \cdot g_2 g_1 \in B_{k-1}',
\end{align*}
 actually, according to Proposition 14, $g_1^2, g_2^2\in B_{k-1}'$ which implies $g_1^2 g_2^2 \in B_{k-1}'$ and $ g_1^2, g_2^2 \in G_{k-1}$ by Proposition \ref{comm_F_k_is_subgroup_of_L_k} also $ g_1^2, g_2^2 \in G_{k-1}$ by induction assumption. From Proposition~\ref{B_k_criteria} it follows that $g_1 g_2  \in G_{k-1}$.

% actually, according Proposition \ref{B'_k and B^2_k} $g_1^2, g_1^2\in B_{k-1}'$ then $g_1^2 g_2^2 \in %B_{k-1}'$ and $ g_1^2, g_2^2 \in G_{k-1}$ by Proposition \ref{B_k_criteria} also $ g_1^2, g_2^2 \in %G_{k-1}$ by induction assumption. From Proposition~\ref{B_k_criteria} it follows that $g_1 g_2  \in %G_{k-1}$ also $ g_1^2, g_2^2 \in G_{k-1}$.

Note that $B_{k-1}' < B_{k-2}$. In other hand $ B_{k-2} <  G_{k-1}$ because $G_{k-1} \simeq  B_{k-2} \ltimes W_{k-2} $ consequently $B_{k-1}' < G_{k-1}$. Besides we have $g_1^2 \in B_{k-1}'$ hence $g_1^2 \in G_{k-1}$.

 Thus, we can use Proposition~\ref{L_k_comm_criteria} (about $G'_k$) from which yields  $g^2 = (g_1^2, g_2^2) \in G_{k}'$.

%$\begin{align*}
%g_1g_2\in L_{k-1} \mbox{ by proposition~\ref{B_k_criteria}}\\
%g_2g_1 = g_1g_2 g_2\m g_1\m g_2 g_1 = g_1g_2 [g_2\m, g_1\m]\in L_{k-1}\mbox{ %proposition~\ref{comm_F_k_is_subgroup_of_L_k}}\\
%g_1 g_2^2 g_1 = g_1^2 g_2^2 [g_2^{-2}, g_1\m] \in F_{k-1}'
%\end{align*}
%\begin{align*}
%g_1g_2\in L_{k-1} \mbox{ by proposition~\ref{B_k_criteria}}\\
%g_2g_1 = g_1g_2 g_2\m g_1\m g_2 g_1 = g_1g_2 [g_2\m, g_1\m]\in L_{k-1}\mbox{ %proposition~\ref{comm_F_k_is_subgroup_of_L_k}}\\
%\end{align*}

Consider the second case $g^2 = (g_1g_2, g_2g_1)$, then $g_1g_2\in G_{k-1}$ \mbox{ by proposition~\ref{B_k_criteria}}. Also second ccodinate $g_2g_1 = g_1g_2 g_2\m g_1\m g_2 g_1 = g_1g_2 [g_2\m, g_1\m]\in G_{k-1}$ \mbox{ by Propositions~\ref{comm_F_k_is_subgroup_of_L_k}} and~\ref{B_k_criteria}.
Analyze the coordinate product of $g^2 = (g_1g_2, \, g_2g_1) $: \, \,

$$g_1g_2 \cdot g_2 g_1 = g_1 g_2^2 g_1 = g_1^2 g_2^2 [g_2^{-2}, g_1\m].$$
%$g_1g_2 \cdot g_2 g_1 = g_1 g_2^2 g_1 = g_1^2 g_2^2 [g_2^{-2}, g_1\m] \in B_{k-1}'$.
 According to Proposition \ref{B'_k and B^2_k} $g_1^2, g_2^2 \in B'_{k-1}$ so we obtaine $g_1^2 g_2^2 [g_2^{-2}, g_1\m] \in B_{k-1}'.$
Thus, $( g_1g_2, g_1g_2) \in G'_{k}$ by Proposition~~\ref{L_k_comm_criteria}.
%\begin{align*}
%g_1g_2\in G_{k-1}   \mbox{ by proposition~\ref{B_k_criteria}}\\
%g_2g_1 = g_1g_2 g_2\m g_1\m g_2 g_1 = g_1g_2 [g_2\m, g_1\m]\in G_{k-1}\mbox{ by propositions~\ref{comm_F_k_is_subgroup_of_L_k} and~\ref{B_k_criteria}}\\
%g_1g_2 \cdot g_2 g_1 = g_1 g_2^2 g_1 = g_1^2 g_2^2 [g_2^{-2}, g_1\m] \in B_{k-1}'
%\end{align*}

\end{proof}

Using this structural property of $(Syl_2 A_{2^k})'$ we deduce a following result.
\begin{theorem}\label{_comm_G_k_eq_[G_k,G_k]} The commutator subgroup $G'_k$ coincides with set of all commutators, other words
$\syl[k]'=\{[f_1,f_2]\mid f_1\in \syl[k], f_2\in \syl[k]\}$.
%Commutator width of $G_k$ is equal to 1.
%The set $G'_k$ coincides with set of all commutators of $G_k$.
\end{theorem}
\begin{proof}
For the case $k=1$ we have $G_1' = \langle e \rangle$, if $k=2$ then $(Syl_2 A_{2^2})' \simeq G_2' \simeq K_4' \langle e \rangle$.   So, further we consider case $k\ge 2$. In order to prove this Theorem we fix an arbitrary element $w\in G_k'$  and then we represent this element as commutator of elements from $G_k$.

We already know by Lemma~\ref{form of comm_2} that every element $w\in\syl[k]'$ we can represent as follow
\begin{align*}
w=(r_1, r_1\m x),
\end{align*}
where  $r_1  \in G_{k-1}$ and $ x \in B'_{k-1}$. %, denote a product $r_1 r_2$ as $ x$, then $r_2=r_1^{-1} x $ which is in form mentioned above, for some $r_1\in \syl[k-1]$.

By proposition~\ref{_comm_F_k_eq_[L_k,F_k]} %(OR \ref{_comm_L_k_eq_[L_k,L_k]} ? )
%about form of commutator $B'_k$
we have $ x = [f,g]$ for some $f\in\aut[k-1]$ and $g\in\syl[k-1]$. Therefore,
\begin{align*}
w=(r_1, r_1\m [f,g]).
\end{align*}

By the Corollary~\ref{c_2_wr_b_elem_repr} we can represent $w$ as commutator of
\begin{align*}
(e,a_{1,2})\sigma \in \aut[k] \mbox{ and } (a_{2,1}, a_{2,2}) \in \aut[k],
\end{align*}
where
$
a_{2,1} = (f\m)^{r_1\m},
a_{2,2} = r_{1} a_{2,1},
a_{1,2} = g^{a_{2,2}\m}.
$

%\begin{align*}
%a_{2,1} &= (f\m)^{r_1\m},\\
%a_{2,2} &= r_{1} a_{2,1},\\
%a_{1,2} &= g^{a_{2,2}\m}.
%\end{align*}

%We note that as $g \in G_{k-1}$ then by proposition~\ref{B_k_criteria} we have $(e,a_{1,2})\sigma \in \syl[k]$.
%So, we have $w = [(e,a_{1,2})\sigma, (a_{2,1}, a_{2,2})]$.

Let us consider the commutator of this elements

\begin{gather*}
\kappa = (e,a_{1,2})\sigma  (a_{2,1}, a_{2,2})  (a_{1,2}\m,e)\sigma\m  (a_{2,1}\m, a_{2,2}\m)
=(a_{1,1}a_{2,2}a_{1,1}\m a_{2,1}\m, a_{1,2}a_{2,1}a_{1,2}\m a_{2,2}\m)\\
=(a_{2,2} a_{2,1}\m, r_1\m  [(a_{2,1}\m)^{r_1}, a_{1,2}^{a_{2,2}}]).
\end{gather*}

Since we transform commutator $\kappa$ in form which is similar to form for $w$. This implies the equations for elements of $\kappa$: $r_1=a_{2,2}a_{2,1}^{-1}$, $f=(a_{2,1}^{-1})^{r_1}$, $g=(a_{1,2})^{a_{2,2}}$.

Let us make sure that this commutator is arbitrary element of form $w=(r_1, r_1\m x)$.
For this goal it is only left to show that $(e,a_{1,2})\sigma, \, \, (a_{2,1}, a_{2,2}) \in G_k$.
Since $ x = [f,g]$ then we have a correspondence $ [f,g]$ to $(a_{2,2} a_{2,1}\m, r_1\m  [(a_{2,1}\m)^{r_1}, a_{1,2}^{a_{2,2}}])$. As a result we obtain:
\begin{eqnarray*}
a_{2,1} &=& (f\m)^{{r_1}\m},\\
a_{2,2} &=& r_{1} a_{2,1},\\
a_{1,2} &=& g^{a_{2,2}\m}\in G_{k-1},\mbox{ by Proposition~\ref{G_k_is_normal_in_B_k} } G_{k-1}\lhd B_{k-1} \mbox{ so } g^{a_{2,2}\m}\in G_{k-1},\\
\mbox{} a_{1,2}^{a_{2,2}} & \in &  G_{k-1}\mbox{  }, \\
a_{2,2} a_{2,1} &=& r_1 a_{2,1}^2\in \syl[k-1], \\
a_{2,1} a_{2,2} &=& a_{2,1} r_1 a_{2,1} = r_1 [r_1, a_{2,1}] a_{2,1}^2 \in \syl[k-1] \mbox{ because } [r_1,a_{2,1}]\in B'_{k-1} \mbox{ by Theorem \ref{_comm_F_k_eq_[L_k,F_k]}}.
\end{eqnarray*} %\ref{_comm_L_k_eq_[L_k,L_k]} %  (NOT \ref{_comm_F_k_eq_[L_k,F_k]})
%In order to use Proposition~\ref{B_k_criteria} we need to show the following
In order to use Proposition~\ref{B_k_criteria} we note that $a_{1,2} = g^{a_{2,2}\m} \in G_{k-1}$  by Proposition~\ref{G_k_is_normal_in_B_k}.
Also $a_{2,1} a_{2,2} = a_{2,1} r_1 a_{2,1} = r_1 [r_1, a_{2,1}] a_{2,1}^2 \in \syl[k-1]$
 by Proposition~\ref{comm_F_k_is_subgroup_of_L_k} and Proposition ~\ref{B'_k and B^2_k}.
%\begin{align*}
%a_{2,2} a_{2,1} &=& r_1 a_{2,1}^2\in \syl[k-1] \\
%a_{1,2} = g^{a_{2,2}\m} &\in G_{k-1}\mbox{ by Proposition~\ref{G_k_is_normal_in_B_k}}.\\
%a_{2,1} a_{2,2} = a_{2,1} r_1 a_{2,1} = r_1 [r_1, a_{2,1}] a_{2,1}^2 &\in \syl[k-1]\mbox{ by %Proposition~\ref{comm_F_k_is_subgroup_of_L_k} and Proposition~\ref{B'_k and B^2_k}}.
%\end{align*}

 %\ref{_comm_L_k_eq_[L_k,L_k]} %  (NOT \ref{_comm_F_k_eq_[L_k,F_k]})
% because of Proposition~\ref{comm_F_k_is_subgroup_of_L_k}
%%($B_{k-1}' < G_{k-1}$)
%and Proposition~\ref{B'_k and B^2_k}
%($B_{k-1}^2 < B_{k-1}'$)
%.

%so conditions of Proposition \ref{comm_F_k_is_subgroup_of_L_k} holds.
% Note that using Proposition~\ref{comm_F_k_is_subgroup_of_L_k} about subgroup of $B'_{k-1}$ into $G_{k-1}$ (in more details $B'_{k-1} \lhd G_{k-1}$) we have inclusion of element $[r_1,a_{2,1}]\in G_{k-1}$, where $r_1\in G_{k-1}$ by condition mentioned above. And $a_{2,1}^2 \in G_{k-1}$ according to Proposition \ref{B'_k and B^2_k}.
%Hence, the product $r_1 [r_1, a_{2,1}] a_{2,1}^2\in G_{k-1}$ too.
%\begin{align*}
%a_{2,1} &=& (f\m)^{(r_1)\m}\\
%a_{2,2} &=& r_{1} a_{2,1}\\
%a_{1,2} &=& g^{a_{2,2}\m}\in G_{k-1}\mbox{ by proposition~\ref{G_k_is_normal_in_B_k}  } G_{k-1}\lhd B_{k-1} %\mbox{ so } g^{a_{2,2}\m}\in G_{k-1}\\
%\mbox{Therefore } a_{1,2}^{a_{2,2}} & \in &  G_{k-1}\mbox{ so conditions of proposition~ %\ref{_comm_F_k_eq_[L_k,F_k]}  holds } \\
%a_{2,2} a_{2,1} &=& r_1 a_{2,1}^2\in \syl[k-1]
%\end{align*}
So we have $(e,a_{1,2})\sigma \in \syl[k]$ and $(a_{2,1}, a_{2,2}) \in \syl[k]$. % it means solution of form $w=(r_1, r_1\m x)$ are confirmed the condition of this Proposition.
\end{proof} % Lemma \ref{comm} (better by

Thus, as it was stated by us in abstract \cite{SkSamara} we obtain the following result.
\begin{corollary}
Commutator width of the group $Syl_2 A_{2^k}$ equal to $1$ for $k\geq 2$.
%Note that it follow from Theorem~\ref{_comm_G_k_eq_[G_k,G_k]} that $cw(Syl_2 A_{2^k})=1$
\end{corollary}

%\begin{corollary}\label{_comm_B_k_eq_[b_k,b_k]}
%$\syl[k]'=\{[f_1,f_2]\mid f_1\in \syl[k], f_2\in \syl[k]\}$.
%Commutator width of $B_k$ is equal to 1.
%The set $B'_k$ coincides with set of all commutators of $B_k$.
%\end{corollary}
%\begin{proof}
%As it was proved in Theorem \ref{max} $G_k \simeq  B_{k-1} \ltimes W_{k-1} $ so $B_{k-1}$ $B_{k-2}$ are subgroups %of $G_k$. According  Theorem 1.2. from \cite{nikolov} we know that $cw G \geq max{cw(A), \frac{1}{n} cw(B)}$ %where $A$ is active, thus $cw( B_{k-2})=1$.
%\end{proof}

\begin{example}
A commutator  of $Syl_2(A_8)$  consist of elements:
$\{ e, (13)(24)(57)(68), (12)(34),\\ (14)(23)(57)(68), (56)(78), (13)(24)(58)(67),(12)(34)(56)(78), (14)(23)(58)(67)\}$.
The commutator $Syl_2 '(A_8) \simeq C_2 ^3$ that is an elementary abelian 2-group of order 8.

\end{example}

  \end{section}

\section{Conclusion }
The commutator width of the wreath product $C_p \wr B$  was founded.
  The commutator width of Sylow 2-subgroups of alternating group ${A_{{2^{k}}}}$, permutation group ${S_{{2^{k}}}}$ and Sylow $p$-subgroups of $Syl_2 A_p^k$ ($Syl_2 S_p^k$) is equal to 1. Commutator width of permutational wreath product $B \wr C_n$, were $B$ is an arbitrary group, was researched. %Its derived subgroup of $G_k$ is nilpotent.
%  \begin{Conclusion}
%  \end{Conclusion}


\begin{thebibliography}{9}

% \bibitem{Mur} \emph{Alexey Muranov,} Finitely generated infinite simple groups of
%infinite commutator width. arXiv:math/0608688v4 [math.GR] 12 Sep 2009.


\bibitem{Mur} \emph{Alexey Muranov.} Finitely generated infinite simple groups of
infinite commutator width. International Journal of Algebra and Computation Vol. 17, (2007). No. 03, pp. 607-659.

%International Journal of Algebra and ComputationVol. 17, No. 03, pp. 607-659 (2007) No Access
%FINITELY GENERATED INFINITE SIMPLE GROUPS OF INFINITE COMMUTATOR WIDTH


\bibitem{Ne}  \emph {V. Nekrashevych.} Self-similar groups. International University Bremen. American Mathematical Society. 2005. Monographs, Vol. 117, 230 p.

\bibitem{Luc}
	 \emph { A. Lucchini.} Generating wreath products and their augmentation ideals. Rend. Sem. Mat. Univ. Padova 98 (1997), 67-87.


\bibitem{Meld}  \emph{J.D.P. Meldrum.} Wreath Products of Groups and Semigroups.  Pitman Monographs and Surveys in Pure and Applied Mathematic. 1st Edition. Jun (1995). 425 p.


\bibitem{nikolov} \emph {Nikolay Nikolov.} On the commutator width of perfect groups. Bull. London Math. Soc. 36 (2004) p. 30–36.

\bibitem{Isacs}  \emph {I. M. Isaacs.} ‘Commutators and the commutator subgroup’, Amer. Math. Monthly 84 (1977)
720–722.

\bibitem{Lav}  \emph{Lavrenyuk Y.} On the finite state automorphism group of a rooted tree Algebra and Discrete Mathematics Number 1. (2002). pp. 79-87

\bibitem{Fite}  \emph{W. B. Fite.} On metabelian groups, Trans. Amer. Math. Soc, 3 (1902), pp. 331-353.

\bibitem{Lin} \emph{Roger C. Lyndon, Paul E. Schupp.}
Combinatorial group theory.
Springer-Verlag Berlin Heidelberg New York 1977. 447 p.


\bibitem{Sam}  \emph{I.V.  Bondarenko, I.O. Samoilovych.}  On finite generation of self-similar groups of finite type.  Int. J. Algebra Comput. February (2013), Volume 23, Issue 01, pp. 69-77

\bibitem{Gr} \emph {R.I. Grigorchuk.} Solved and unsolved problems around one group. Infinite Groups: Geometric, Combinatorial and Dynamical Aspects. Вasel, (2005). Progress Math., vol 248. pp. 117-218.

%\bibitem{Olij} \emph{Olijnuk A. S.} Isomorphic images of amalgamated products infinite cyclic groups of finite states automorphisms of the p-adic root tree. Reports of the National Academy of Sciences of Ukraine, 2010, №12   P. 20- 24.

%\bibitem{Shar} \emph {V. V. Sharko.} Smooth topological equivalence of functions of surfaces. Ukrainian Mathematical Journal, May (2003), Volume 55, Issue 5,  pp. 832–846.

\bibitem{Kal}  \emph {
L. Kaloujnine.} "La structure des p-groupes de Sylow des groupes
symetriques finis", Annales Scientifiques de l'Ecole Normale
Superieure. Troisieme Serie 65, (1948) pp. 239–276.

\bibitem{Paw}
\emph {B. Pawlik.} The action of Sylow 2-subgroups of symmetric groups on the set of bases and the problem of isomorphism of their Cayley graphs. Algebra and Discrete Mathematics. (2016), Vol. 21, N. 2, pp. 264-281.

\bibitem{Sk}
\emph {R.~Skuratovskii,} "Corepresentation of a Sylow p-subgroup of a group Sn". Cybernetics and systems analysis,  (2009),  N. 1,  pp. 27-41.

%\bibitem{Sk} \emph{R.~Skuratovskii}. Generators and relations for sylows р-subgroup of group $S_n$. \emph{Naukovi Visti KPI.} \textbf{4} (2013), pp. 94--105. (in Ukrainian)

   \bibitem{DrSku} \emph{R.V. Skuratovskii, Y.A. Drozd, } Generators and and relations for wreath products of groups.
     Ukr Math J. (2008), vol. 60. Issue 7, pp. 1168–1171.


%\bibitem{SkThes} \emph {R. V. Skuratovskii,} Minimal generating systems and properties of $Syl_2A_{2^k}$ and $Syl_2A_{n}$. X International Algebraic Conference in Odessa dedicated to the 70th anniversary of Yu. A. Drozd. (2015), pp. 104.

%\bibitem{SkThes2} \emph {R. V. Skuratovskii,} Minimal generating systems and structure of $Syl_2A_{2^k}$ and $Syl_2A_{n}$. International Conference and PhD-Master Summer School on
% Graphs and Groups, Spectra and Symmetries. (2016), source: http://math.nsc.ru/conference/g2/g2s2/exptext/Skuratovskii-abstract-G2S2+.pdf.


\bibitem{SkComm} \emph {R. V. Skuratovskii,}
The commutator and centralizer description of Sylow
subgroups of alternating and symmetric groups. Source: https://arxiv.org/pdf/1712.01401.pdf

\bibitem{SkIrred}	Skuratovskii R. V.    Involutive irreducible generating sets and structure of sylow 2-subgroups of alternating groups.  ROMAI J., \textbf{13},
      ROMAI Journal.  Issue 1, (2017), p117-139.

%\bibitem{SkThes3} \emph {R. V. Skuratovskii,} Structure of commutant and centralizer of Sylow 2-subgroups of
%alternating and symmetric groups, minimal generating sets of Syl2An, its
%applications in cryptography. International Conference and PhD-Master Summer School on
% Graphs and Groups, Spectra and Symmetries. (2016), source:

\bibitem{SkAr} \emph {R. V. Skuratovskii,} Structure and minimal generating sets of Sylow 2-subgroups of alternating groups. Source: https://arxiv.org/abs/1702.05784v2

%\bibitem{SkThes4} \emph {R. V. Skuratovskii,} Structure of commutant and centralizer of Sylow 2-subgroups of alternating and symmetric groups, minimal generating sets of $Syl_2 A_n$, its applications in cryptography
% Romanian. Conference CAIM 2017.
%Source: https://www.romai.ro/conferintele$_$romai/caim2017/list$_$of$_$participants.php

\bibitem{SkThes5} \emph {R. V. Skuratovskii,}
Structure of commutant and centralizer, minimal generating sets of. Sylow 2-subgroups $Syl_2 A_n$  of alternating and symmetric groups. International conference in Ukraine, ATA12. (2017).  https://www.imath.kiev.ua/~topology/.../skuratovskiy.pdf

\bibitem{SkOdes} \emph {R. V. Skuratovskii, } The commutator and centralizer of Sylow subgroups of alternating
and symmetric groups, its minimal generating set.  International Scientific Conference. «Algebraic and geometric
methods of analysis». (2018) P. 58.
	

\bibitem{SkSamara} \emph {R. V. Skuratovskii,} The commutator of Sylow 2-subgroups of alternating and symmetric groups.
 The Seventh School-Conference on. Lie Algebras, Algebraic Groups and Invariant Theory. Samara, Russia. (2018) pp. 42-43.
 Source: algeom.samsu.ru/2018-samara/2018/2018.Thesis.pdf

%\bibitem{SkSyl} \emph {R. V. Skuratovskii,} {\it Structure and minimal generating sets of Sylow 2-subgroups of alternating groups.} Sao Paulo Journal of Mathematical Sciences. (2018), no. 1,  pp. 1-19.
%        Source: https://link.springer.com/article/10.1007/s40863-018-0085-0.

\bibitem{DrSk}
Drozd Y.A., Skuratovskii R.V. Generators and relations for wreath products.// Ukr. math. journ. - 60, n. 7. - 2008.- S. pp. 997-999.

%\bibitem{Gural} Guralnick R. M. Expressing group elements as commutators. Journal of Mathematics.  Volume 10, Number 3, Summer (1980). pp. 651-654.

\bibitem{Dm}
\emph {U.~Dmitruk, V.~Suschansky,}  Structure of 2-sylow subgroup of alternating group and normalizers of symmetric and alternating group. UMJ. (1981), N. 3,  pp. 304-312.

\bibitem{Kar} \emph{ M. I. Kargapolov, J. I. Merzljakov,}
Fundamentals of the Theory of Groups Springer. Softcover reprint of the original 1st ed. 1979 edition. Springer. 1st ed. 1979. 312 p.

\bibitem{Heinek}
\emph {H. Heineken}  Normal embeddings of p-groups into p-groups. Proc. Edinburgh Math.
Soc. 35 (1992), pp. 309-314.


\end{thebibliography}
\end{document}